\documentclass[11pt]{article}
\usepackage[centertags]{amsmath}
\usepackage{amsfonts}
\usepackage{amssymb}
\usepackage{amsthm}
\usepackage{newlfont}

\newtheorem{Theorem}{Theorem}[section]
\newtheorem{Definition}[Theorem]{Definition}
\newtheorem{Proposition}[Theorem]{Proposition}
\newtheorem{Lemma}[Theorem]{Lemma}

\newtheorem{Remark}[Theorem]{Remark}

\newtheorem{Hypothesis}{Hypothesis}

\numberwithin{equation}{section}

\begin{document}

\def\le{\left}
\def\r{\right}
\def\cost{\mbox{const}}
\def\a{\alpha}
\def\d{\delta}
\def\ph{\varphi}
\def\e{\epsilon}
\def\la{\lambda}
\def\si{\sigma}
\def\La{\Lambda}
\def\B{{\cal B}}
\def\A{{\mathcal A}}
\def\L{{\mathcal L}}
\def\O{{\mathcal O}}
\def\bO{\bar{{\mathcal O}}}
\def\F{{\mathcal F}}
\def\K{{\mathcal K}}
\def\H{{\mathcal H}}
\def\D{{\mathcal D}}
\def\C{{\mathcal C}}
\def\M{{\mathcal M}}
\def\N{{\mathcal N}}
\def\G{{\mathcal G}}
\def\T{{\mathcal T}}
\def\R{{\mathcal R}}
\def\I{{\mathcal I}}

\def\bw{\bar{W}}
\def\phin{\|\varphi\|_{0}}
\def\s0t{\sup_{t \in [0,T]}}
\def\lt{\lim_{t\rightarrow 0}}
\def\iot{\int_{0}^{t}}
\def\ioi{\int_0^{+\infty}}
\def\ds{\displaystyle}
\def\pag{\vfill\eject}
\def\fine{\par\vfill\supereject\end}
\def\acapo{\hfill\break}

\def\beq{\begin{equation}}
\def\eeq{\end{equation}}
\def\barr{\begin{array}}
\def\earr{\end{array}}
\def\vs{\vspace{.1mm}   \\}
\def\rd{\reals\,^{d}}
\def\rn{\reals\,^{n}}
\def\rr{\reals\,^{r}}
\def\bD{\bar{{\mathcal D}}}
\newcommand{\dimo}{\hfill \break {\bf Proof - }}
\newcommand{\nat}{\mathbb N}
\newcommand{\E}{\mathbb E}
\newcommand{\Pro}{\mathbb P}
\newcommand{\com}{{\scriptstyle \circ}}
\newcommand{\reals}{\mathbb R}

\title{Averaging principle for a class of stochastic reaction-diffusion equations  }

\author{Sandra Cerrai\\
Dip. di Matematica per le Decisioni\\
Universit\`{a} di Firenze\\
Via C. Lombroso 6/17\\
 I-50134 Firenze, Italy
  \and
Mark Freidlin\\
Department of Mathematics\\
 University of Maryland\\
College Park\\
 Maryland, USA}

\date{}

\maketitle

\begin{abstract}
We consider the averaging principle for stochastic
reaction-diffusion equations. Under some assumptions providing
existence of a unique invariant measure of the fast motion with
the frozen slow component, we calculate limiting slow  motion. The
study of solvability of Kolmogorov equations  in Hilbert spaces
and the analysis of regularity properties of solutions, allow to
generalize the classical approach to finite-dimensional problems
of this type in the case of SPDE's.

\end{abstract}

\section{Introduction}
\label{intro}

Consider a Hamiltonian system with one degree of freedom. In the
area where the Hamiltonian has no critical points, one can
introduce action-angle coordinates $(I,\varphi)$, with  $I \in
\mathbb{R}^{1}$ and  $0\leq \varphi \leq 2\pi$, so that the system
has the form
\begin{equation}
\dot{I}_{t}=0, \hspace{0.4cm}
\dot{\varphi}_{t}=\omega(I_{t}).\label{ActionAngle1}
\end{equation}

Now, consider small perturbations of this system such that, after
an appropriate time rescaling, the perturbed system can be written
as follows

\begin{equation}
\dot{I}_{t}^{\epsilon}=\beta_{1}(I_{t}^{\epsilon},\varphi_{t}^{\epsilon}),
\hspace{0.4cm}
\dot{\varphi}_{t}^{\epsilon}=\frac{1}{\epsilon}\,\omega(I_{t}^{\epsilon})+\beta_{2}(I_{t}^{\epsilon},\varphi_{t}^{\epsilon}).\label{ActionAnglePerturbed}
\end{equation}
Here the perturbations $\beta_{1},\beta_{2}:\mathbb{R}^1\times
[0,2\pi]\to \mathbb{R}$ are assumed to be regular enough
functions, as well as $\omega:\mathbb{R}\to \mathbb{R}$, and
$0<\epsilon<<1$.

System (\ref{ActionAnglePerturbed}) has a fast component, which
is, roughly speaking, the motion along the non-perturbed
trajectories (\ref{ActionAngle1}), after the time change
$t\rightarrow t/\epsilon$, and the slow component which can be
described by the evolution of $I_{t}^{\epsilon}$. When $\epsilon$
goes to $0$, the slow component approaches the averaged motion
$\bar{I}_{t}$, defined by

\begin{equation}
\dot{\bar{I}}_{t}=\bar{\beta}_{1}(\bar{I}_{t}),\ \
\ \ \ \bar{I}_{0}=I_{0},\label{AveragedSlowMotion}
\end{equation}
where
\[\bar{\beta}_{1}(y)=\frac{1}{2\pi}\int_{0}^{2\pi}\beta_{1}(y,\varphi)d\varphi.\]
This is a classical manifestation of the averaging principle for
equation (\ref{ActionAnglePerturbed}).

To prove the convergence of $I^{\epsilon}_{t}$ to
$\bar{I}_{t}$, one can consider a $2\pi$-periodic in
$\varphi$ solution $u(I,\varphi)$ of an auxiliary equation

\begin{equation}
\mathcal{L}^Iu(I,\varphi):=\omega(I)\frac{\partial u}{\partial \varphi}=\beta_{1}(I,\varphi)-\bar{\beta}(I).
\label{AuxiliaryEquation}
\end{equation}

It is easy to see that such a solution exists and is unique up to
an additive function, depending just on $I$. Moreover, it can be
chosen in such a way that $u(I,\varphi)$ has continuous
derivatives in $I$ and $\varphi$. Actually, $u(I,\varphi)$ can be
written explicitly. It follows from (\ref{ActionAnglePerturbed})
and (\ref{AuxiliaryEquation}) that
\[\begin{array}{l}
\ds{u(I_{t}^{\epsilon},\varphi_{t}^{\epsilon})-u(I_{0}^{\epsilon},\varphi_{0}^{\epsilon})=
\frac{1}{\epsilon}\int_{0}^{t}\frac{\partial u}{\partial
\varphi}\,(I_{s}^{\epsilon},\varphi_{s}^{\epsilon})\omega(I_{s}^{\epsilon})ds+ \int_{0}^{t}\frac{\partial u}{\partial
\varphi}\,(I_{s}^{\epsilon},\varphi_{s}^{\epsilon})\beta_{2}(I_{s}^{\epsilon},\varphi_{s}^{\epsilon})ds}\\
\vs \ds{ +\int_{0}^{t}\frac{\partial
u}{\partial
I}\,(I_{s}^{\epsilon},\varphi_{s}^{\epsilon})\beta_{1}(I_{s}^{\epsilon},\varphi_{s}^{\epsilon})ds=\frac
1\epsilon\,\int_{0}^{t}[\beta_{1}(I_{s}^{\epsilon},\varphi_{s}^{\epsilon})-\bar{\beta}(I_{s}^{\epsilon})]ds}\\
\vs \ds{ +
\int_{0}^{t}\frac{\partial u}{\partial
\varphi}\,(I_{s}^{\epsilon},\varphi_{s}^{\epsilon})\beta_{2}(I_{s}^{\epsilon},\varphi_{s}^{\epsilon})ds+\int_{0}^{t}\frac{\partial u}{\partial
I}\,(I_{s}^{\epsilon},\varphi_{s}^{\epsilon})\beta_{1}(I_{s}^{\epsilon},\varphi_{s}^{\epsilon})ds.}
\end{array}\]
Hence, by taking into account the boundedness of coefficients $\beta_1$ and $\beta_2$ and of function $u(I,\varphi)$
together with its first derivatives, one can conclude from the last
equality that for any $T>0$
\begin{equation}
\sup_{0\leq t\leq
T}|\int_{0}^{t}[\beta_{1}(I_{s}^{\epsilon},\varphi_{s}^{\epsilon})-\bar{\beta}(I_{s}^{\epsilon})]ds|\leq
c\,\epsilon,\label{convergence1}
\end{equation}
for some constant $c>0$. Now, from (\ref{ActionAnglePerturbed}) and
(\ref{AveragedSlowMotion}) it follows
\[
I^{\epsilon}_{t}-\bar{I}_{t}=\int_{0}^{t}[\beta_{1}(I_{s}^{\epsilon},\varphi_{s}^{\epsilon})-\bar{\beta}(I_{s}^{\epsilon})]ds+
\int_{0}^{t}[\bar{\beta}(I_{s}^{\epsilon})-\bar{\beta}(\bar{I}_{s})]ds,\]
so that, assuming that $\beta(I,\varphi)$ (and thus
$\bar{\beta}(I)$) is Lipschitz-continuous, thanks to
(\ref{convergence1}) and to 	Gronwall's lemma we get
\begin{displaymath}
\sup_{0\leq t\leq T}|I_{t}^{\epsilon}-\bar{I}_{t}|\leq
c\,\epsilon.
\end{displaymath}

On a first glance, one can think that consideration of the
auxiliary equation (\ref{AuxiliaryEquation}) for proving averaging
principle is an artificial trick. But, actually, this is not the
case; the use of equation (\ref{AuxiliaryEquation}) and its natural
generalizations helps to prove averaging principle in many cases.
For example, when deterministic perturbations of a completely
integrable system with many degrees of freedom (in a domain where
one can introduce action-angle coordinates) are considered, the
operator $\mathcal{L}$ is the generator of the corresponding flow
on a torus. Because of the existence of resonance tori, where
invariant measure of the flow is not unique, one has to consider
approximate solutions of the corresponding equation
(\ref{AuxiliaryEquation}). The price for this is that the
convergence of $\sup_{0\leq t\leq
T}|I_{t}^{\epsilon}-\bar{I}_{t}|$ to zero does not hold for
any fixed initial condition, but just in Lebesgue measure in the
phase space, given that the set of resonance tori is small enough
(see \cite{neistadt}). An approximate solution of the
corresponding analogue of equation (\ref{AuxiliaryEquation}) is
used in \cite{frven2} for averaging of stochastic perturbations.
In this case it is possible to prove weak convergence to the
averaged system in the space of continuous functions on the phase
space.
Moreover, concerning the use of the auxiliary equation \eqref{AuxiliaryEquation},  it is worthwhile mentioning that in \cite{papanicolaou} suitable {\em correction functions} arising as solutions of problems analogous to \eqref{AuxiliaryEquation} are introduced in order to prove some limit theorems for more general multi-scaling systems.

An analogue of equation (\ref{AuxiliaryEquation}) appears also in
the case when the fast motion is a stochastic process
\begin{displaymath}
\dot{I}_{t}^{\epsilon}=\beta_{1}(I_{t}^{\epsilon},\varphi_{t}^{\epsilon}),\hspace{0.2cm}
\dot{\varphi}_{t}^{\epsilon}=\frac{1}{\epsilon}\,\omega(I_{t}^{\epsilon},\varphi_{t}^{\epsilon})+
\frac{1}{\sqrt{\epsilon}}\,\sigma(I_{t}^{\epsilon},\varphi_{t}^{\epsilon})\dot{w}_{t}+
\beta_{2}(I_{t}^{\epsilon},\varphi_{t}^{\epsilon}).
\end{displaymath}
Here $I, \varphi:[0,+\infty)\to \reals^n$, 
$\omega:\reals^n\times \reals^n\to \reals^n$, 
$\sigma(I,\varphi)\sigma^{\ast}(I,\varphi)=\alpha(I,\varphi)$ is a
positive definite $n \times n-$matrix and  $w_{t}$ is the
standard $n$-dimensional Wiener process. All functions are assumed
to be $2\pi-$periodic in the variables $\varphi_{i}$ and smooth enough. Under these conditions, for each $I\in
\mathbb{R}^{n}$ the diffusion process $\varphi_{t}^{I}$ on the
$n$-torus $T^{n}$ defined by the equation
\[\dot{\varphi}_{t}^{I}=\omega(I, \varphi_{t}^{I})+\sigma(I,
\varphi_{t}^{I})\dot{w}_{t},\]
 has a unique invariant measure
with  density $m_{I}(\varphi)$. Then  equation
(\ref{AuxiliaryEquation}) should be replaced by

\begin{equation}
\label{AuxiliaryEquationReplaced}
\mathcal{L}^{I}u(I,\varphi)=\beta_{1}(I,\varphi)-\bar{\beta}_{1}(I),\end{equation}
where $\mathcal{L}^{I}$ is the generator of the process
$\varphi_{t}^{I}$ and for any $I\in \mathbb{R}^{n}$
\[\bar{\beta}_{1}(I):=\int_{T^{n}}\beta_{1}(I,\varphi)m_{I}(\varphi)d\varphi.\]

Taking into account the uniqueness of the invariant measure, one can
check that there exists a solution to problem
(\ref{AuxiliaryEquationReplaced}) which is smooth in $I$ and
$\varphi$. Applying It\^o's formula to
$u(I_{t}^{\epsilon},\varphi_{t}^{\epsilon})$, one can prove not
just weak convergence of $I_{t}^{\epsilon}$ to $\bar{I}_{t}$
on any finite time interval, but also convergence of
$(I_{t}^{\epsilon}-\bar{I}_{t})/\sqrt{\epsilon}$ to a
diffusion process.

Besides the situations described above, averaging principle both for deterministically and for randomly perturbed systems, having a finite number of degrees of freedom, has been studied by many authors, under different assumptions and with different methods. The first rigorous results are due to Bogoliubov
(see \cite{bogo}). Further developments were obtained by Volosov, Anosov and
Neishtadt (see \cite{neistadt} and
\cite{volosov}) and  by Arnold et al. (see \cite{akn}). All these references are for the deterministic case. Concerning the stochastic case, it is worth quoting the paper by Khasminskii \cite{khas}, the works of Brin, Freidlin and Wentcell  (see \cite{brfr},
\cite{freidlin}, \cite{frven}, \cite{frven2}), Veretennikov  (see \cite{vere}) and
Kifer (see for example \cite{kif1}, \cite{kif2},
\cite{kif3}, \cite{kif4}).

\medskip

To the best of our knowledge, very few has been done as far as averaging for infinite dimensional systems is concerned. To this purpose we recall the papers \cite{massei} and \cite{seivr}, where the case of stochastic evolution equations in abstract Hilbert spaces is considered, and the paper \cite{kukpia}, where randomly perturbed KdV equation is studied. 

In this paper we are dealing with a system of
reaction-diffusion equations with a stochastic fast component.
Namely, for each $0<\e<<1$, we consider the system of partial
differential equations
\begin{equation}
\label{sistema} \le\{
\begin{array}{l}
\ds{ \frac{\partial u^\e}{\partial t}(t,\xi)=\mathcal{A}
u^\e(t,\xi)+f(\xi,u^\e(t,\xi),v^\e(t,\xi)),  \ \ \ \ \ t\geq 0,\ \ \ \xi \in\,[0,L],}\\
\vs \ds{\frac{\partial v^\e}{\partial t}(t,\xi)=\frac
1\e\,\le[\mathcal{B}
v^\e(t,\xi)+g(\xi,u^\e(t,\xi),v^\e(t,\xi))\r]+\frac
1{\sqrt{\e}}\,\frac{\partial w}{\partial t}(t,\xi), \ \ \ \ \
t\geq 0,\ \ \ \xi \in\,[0,L],}\\
\vs \ds{u^\e(0,\xi)=x(\xi),\ \ \ \ v^\e(0,\xi)=y(\xi),\ \ \ \ \
\xi
\in\,[0,L],}\\
\vs \ds{ \mathcal{N}_{1} u^\e\,(t,\xi)=\mathcal{N}_{2}
v^\e\,(t,\xi)=0,\ \ \ \ t\geq 0,\ \ \ \ \xi \in\,\{0,L\}.}
\end{array}\r.
\end{equation}
The present  model describes a typical and relevant situation for reaction-diffusion systems in which the diffusion coefficients and the rates of reactions have
different order. In the case we are considering  here, the noise is included  only in the fast motion and it is of additive type. However, we would like to stress that the introduction of a noisy term of additive type in the slow equation  would not lead to
any new effects, as it  should be
included in the limiting slow motion without any substantial changes.

\medskip

The  linear operators $\mathcal{A}$ and $\mathcal{B}$, appearing
respectively in  the slow and in  the  fast equation, are
second order uniformly elliptic operators and $\mathcal{N}_1$ and
$\mathcal{N}_2$ are some operators acting on the boundary. The
operator $\mathcal{B}$, endowed with the boundary conditions
$\mathcal{N}_2$, is self-adjoint and strictly dissipative (see
Hypothesis \ref{H1}).

The reaction coefficients $f$ and $g$ are measurable mappings from
$[0,L]\times \mathbb{R}^2$ into $\mathbb{R}$ which satisfy
suitable regularity assumptions and for the reaction coefficient
$g$ in the fast motion equation some  dissipativity assumption is
assumed (see Hypotheses \ref{H2} and \ref{H3}).

The noisy perturbation of  the fast motion equation  is given  by
a space-time white noise $\partial w/\partial t(t,\xi)$, defined
on a complete stochastic basis $(\Omega, \mathcal{F},
\mathcal{F}_t, \mathbb{P})$.

\medskip

The corresponding fast motion $v^{x,y}(t)$, with frozen slow
component $x \in\,H:=L^2(0,L)$, (the counterpart of the process
$\varphi_{t}^{I}$ above in the case of a system with a finite
number of degrees of freedom) is a Markov process in a functional
space. Notice that the phase space of $v^{x,y}(t)$ is not just
infinitely dimensional but also not compact. Nevertheless, by
assuming that the system has certain dissipativity properties, for
any fixed $x \in\,H$ the process $v^{x,y}(t)$ has a unique
invariant measure $\mu^{x}$. If $\mathcal{L}^{x}$ is the generator
of this process, then the counterpart of equation
\eqref{AuxiliaryEquation} has the form

\begin{equation}
\label{counterpart}
c(\e)\,\Phi_h^{\e}(x,y)-\mathcal{L}^{x}\Phi_h^{\e}(x,y)=\le<F(x,y)-
\bar{F}(x),h\r>_{H},\ \ \ \ \ x,y,h \in\,H,
\end{equation}
where $c(\e)$ is a constant depending on $\e$ and vanishing at
$\e=0$, \[F(x,y)(\xi)=f(\xi, x(\xi),y(\xi)),\ \ \ \ \ \xi
\in\,[0,L],\] and
\[\bar{F}(x):=\int_{H} F(x,y)\,\mu^x(dy),\ \ \ \ x
\in\,H.\]
Notice  that in \eqref{counterpart} we
cannot consider the Poisson equation ($c(\e)=0$), but we have to
add a zero order term $c(\e)\Phi^\e_h(x,y)$, in order to get  bounds for $\Phi^\e_h(x,y)$ and its derivatives which are uniform with respect to $\e \in\,(0,1]$.

Due to the ergodicity of $\mu^x$, we prove that there exists some $\d>0$ such that for any
$\varphi:H\to \mathbb{R}$ and $x,y \in\,H$
\[\le|P^x_t\varphi(y)-\int_H \varphi(z)\,\mu^x(dz)\r|\leq
c\,\le(1+|x|_H+|y|_H\r)\,e^{-\d t}\,[\varphi]_{\text{\tiny
Lip}},\] where $P^x_t$ is the transition semigroup associated with
the fast motion $v^{x,y}(t)$, with frozen slow component $x$. This
implies that the solution $\Phi^\e_h(x,y)$ of equation
\eqref{counterpart} can be written explicitly as
\[\Phi^\e_h(x,y)=\int_0^t e^{-c(\e)t} P^x_t\le<F(x,\cdot)-
\bar{F}(x),h\r>_{H}(y)\,dt.\]

By using some techniques developed in \cite{cerrai}, we obtain  bounds for
$\Phi^\e_h(x,y)$ and its derivatives, which in general are not
uniform in $x,y \in\,H$, as the reaction coefficients $f$ and $g$
are not assumed to be bounded. Moreover, we are able to apply an
infinite dimensional It\^o's formula to $\Phi^\e_h(u^\e,P_n
v^\e)$, where $P_n$ is the projection of $H$ onto the
$n$-dimensional space generated by the first $n$ modes of the operator $\cal{B}$, and $u^\e$
and $v^\e$ are the solutions of system \eqref{sistema}. In this
way, as in the case of a system with a finite number of degrees of
freedom, we are able to prove that
\begin{equation}
\label{re} \E\,\le|\int_0^t\le<F(u^\e(s),v^\e(s))-
\bar{F}(u^\e(s)),h\r>_{H}\,ds\r|\leq K_t(\e),\ \ \ \ \ t\geq 0,\ \
\e>0, \end{equation}
 for some $K_t(\e)\downarrow 0$,  as  $\e$ goes to
zero. The proof of \eqref{re}  is one of the major task of the
paper, as it requires several   estimates for $\Phi^\e_h(x,y)$ and
its derivatives and  uniform bounds with respect to  $\e>0$, both for $u^\e$ and
for $v^\e$.

Once we have estimate \eqref{re}, we show that  for any $T>0$ the
family $\{\mathcal{L}(u^\e)\}_{\e \in\,(0,1)}$ is tight in
$\mathcal{P}(C((0,T];H)\cap L^\infty(0,T;H))$ and then
 we identify the weak limit of any subsequence
of $\{u^\e\}$ with the solution $\bar{u}$ of the averaged equation
\begin{equation}
\label{averagedbis} d\bar{u}(t)=A \bar{u}(t)+\bar{F}(\bar{u}(t)),\
\ \ \ \ \bar{u}(0)=x \in\,H. \end{equation}
Now, as a consequence of   the dependence of
$\mu^x$   on $x \in\,H$,  the nonlinear term $\bar{F}$ in
\eqref{averagedbis} is a functional of $\bar{u}$. Nevertheless,
one can prove that problem \eqref{averagedbis}, under certain
small assumptions, has a unique solution (see section \ref{sec5},
Proposition \ref{existence}).
Hence, by a uniqueness
argument, this allows us to conclude that the whole sequence $\{ u^\e\}_{\e>0}$ converges to   $\bar{u}$ in probability, uniformly on any finite time interval $[0,T]$. That is 
\begin{Theorem}
Under Hypotheses \ref{H1}, \ref{H2} and \ref{H3} (see Section 2), for any $T>0$, $x, y \in\,W^{\a,2}(0,L)$, with $\a>0$,  and $\eta>0$ it holds
\[\lim_{\e\to 0}\,\Pro \le(\,\sup_{t \in\,[0,T]}|u^\e(t)-\bar{u}(t)|_{H}>\eta\,\r)=0.\]

\end{Theorem}

Notice that in the case of space dimension $d=1$ the fast equation with frozen slow component is a gradient system and then we have an explicit expression for the invariant measure $\mu^x$. This allows to prove that the mapping $\bar{F}$ is differentiable and to give an expression for its derivative. In such a way we can study dependence with respect to $x$ for the correction function $\Phi_h^\e$ and we can proceed with the use of It\^o's formula. 
In the case of space dimension $d>1$ the fast equation is no more of gradient type. Nevertheless,  in Section \ref{sec6} we show how it is possible to overcome this difficulty and how, by a suitable approximation procedure, it is still possible to prove averaging.

\medskip

Finally, we would like to recall that in a number of models, one can assume that the noise in the fast
motion  is small. This results in replacement of $\e^{-1/2}$ by
$\d\,\e^{-1/2}$, with $0<\d<<1$ (to this purpose we refer to
\cite{freidlin}). Then, in generic situation, the invariant
measure of the fast motion with frozen slow component $x$  is
concentrated, as $\d$ goes to zero, near one point $y^\star(x)
\in\,H$. This is a result of large deviations bounds and
$y^\star(x)$ can be found as an extremal of a certain functional.
In particular, if the operator $\mathcal{B}$ in the fast equation
is self-adjoint and $g(\xi,\si,\rho)=h(\si) N(\rho)$, with, for
brevity, the antiderivative $H(\si)$ of $h(\si)$ having a unique
maximum point, then $y^\star(x)$ is a constant providing the
maximum of $H(\si)$. In this case we have that
$\bar{F}(x)(\xi)=f(\xi,x(\xi),y^\star(x)(\xi))$, $\xi
\in\,[0,\pi]$, and \eqref{averagedbis} is a classical
reaction-diffusion equation. We will address this problem
somewhere else.

\section{Assumptions and notations}
\label{sec2}

Let $H$ be the Hilbert  $L^2(0,L)$, endowed with the usual scalar
product $\le<\cdot,\cdot\r>_H$ and the corresponding norm
$|\cdot|_H$. In what follows, we shall denote by $\mathcal{L}(H)$
the Banach space of bounded linear operators on $H$, endowed with
the usual sup-norm. $\mathcal{L}_1(H)$ denotes the Banach space of
{\em trace-class} operators, endowed with the norm
\[\|A\|_1:=\mbox{Tr}\,[\sqrt{A A^\star}],   \]
and $\mathcal{L}_2(H)$ denotes the Hilbert space of {\em
Hilbert-Schmidt} operators on $H$, endowed with the scalar product
\[\le<A,B\r>_{2}=\mbox{Tr}\,[AB^\star]\]
and the corresponding norm $\|A\|_2=\sqrt{\mbox{Tr}\,[AA^\star]}$.
\medskip

The  Banach space of bounded Borel functions $\varphi:H\to
\reals$, endowed with the sup-norm
\[\|\varphi\|_0:=\sup_{x \in\,H}|\varphi(x)|,\]
will be denoted by $B_b(H)$. $C_b(H)$ is the subspace of all
uniformly continuous mappings and $C^k_b(H)$ is the subspace of
all $k$-times differentiable mappings,  having bounded and
uniformly continuous derivatives, up to the $k$-th order, for $k
\in\,\nat$. $C^k_b(H)$ is a Banach space endowed with the norm
\[\|\varphi\|_k:=\|\varphi\|_0+\sum_{i=1}^k\sup_{x
\in\,H}|D^i\varphi(x)|_{\mathcal{L}^i(H)}=:\|\varphi\|_0+\sum_{i=1}^k\,[\varphi]_i,\]
where $\mathcal{L}^1(H):=H$ and, by recurrence,
$\mathcal{L}^i(H):=\mathcal{L}(H,\mathcal{L}^{i-1}(H))$, for any
$i>1$.

In what follows we shall denote by $\text{Lip}(H)$ the set of all
Lipschitz-continuous mappings $\varphi:H\to \reals$ and we shall
set
\[[\varphi]_{\text{\tiny
Lip}}:=\sup_{x\neq y} \frac{|\varphi(x)-\varphi(y)|}{|x-y|}.\]
Moreover, for any $k\geq 1$ we shall denote by $\text{Lip}^k(H)$
the subset of all $k$-times differentiable mappings  having
bounded and uniformly continuous derivatives, up to the $k$-th
order. Notice that for any $\varphi \in\,\text{Lip}(H)$ 
\begin{equation}
\label{fi} |\varphi(y)|\leq [\varphi]_{\text{\tiny
Lip}}|y|_H+|\varphi(0)|,\ \ \ \ \ y \in\,H.
\end{equation}

\medskip

The stochastic perturbation in the fast motion equation is given
by a space-time white noise $\partial w/\partial t(t,\xi)$, for $t \geq 0$ and $\xi
\in\,[0,L]$. Formally the cylindrical Wiener process  $w(t,\xi)$ is
defined as the infinite sum
\[w(t,\xi)=\sum_{k=1}^\infty e_k(\xi)\,\beta_k(t),\ \ \ \ \ t\geq
0,\ \ \ \xi \in\,[0,L],\] where $\{e_k\}_{k \in\,\nat}$ is a
complete orthonormal basis in $H$ and $\{\beta_k(t)\}_{k
\in\,\nat}$ is a sequence of mutually independent standard
Brownian motions defined on the same complete stochastic basis
$(\Omega,\mathcal{F}, \mathcal{F}_t, \mathbb{P})$.

Now, for any $T>0$ and $p\geq 1$ we shall denote by
$\mathcal{H}_{T,p}$ the space of adapted processes in
$C((0,T];L^p(\Omega;H))\cap L^\infty(0,T;L^p(\Omega;H))$. $\mathcal{H}_{T,p}$ is a Banach space,
endowed with the norm
\[\|u\|_{\mathcal{H}_{T,p}}=\le(\,\sup_{t
\in\,[0,T]}\,\E\,|u(t)|_H^p\r)^{\frac 1p}.\] Moreover, we shall
denote by $\mathcal{C}_{T,p}$ the subspace of processes $u
\in\,L^p(\Omega;C((0,T];H)\cap L^\infty(0,T;H))$, endowed with the norm
\[\|u\|_{\mathcal{C}_{T,p}}=\le(\,\E\,\sup_{t
\in\,[0,T]}\,|u(t)|_H^p\r)^{\frac 1p}.\]

\bigskip

The  linear operators $\mathcal{A}$ and $\mathcal{B}$, appearing
respectively in the slow and in the fast motion equation, are
second order uniformly elliptic operators, having continuous  coefficients on $[0,L]$, and   $\mathcal{N}_1$ and $\mathcal{N}_2$
can be either the identity operator (Dirichlet boundary
conditions) or first order operators of the following type
\[\beta(\xi) \frac{\partial}{\partial \xi}+\gamma(\xi),\ \
\ \ \xi \in\,\{0,L\},\] for some  $\beta, \gamma \in\,C^1[0,L]$
such that $\beta(\xi)\neq 0$,  for $\xi=0,L$.

As known,  the realizations $A$ and $B$ in $H$ of the second order
operators $\mathcal{A}$ and $\mathcal{B}$, endowed respectively
with the boundary condition $\mathcal{N}_1$ and $\mathcal{N}_2$,
generate two analytic semigroups, which will be denoted by  $e^{t
A}$ and $e^{t B}$, $t\geq 0$. Their domains $D(A)$ and $D(B)$ are
given by
\[W^{2,2}_{\mathcal{N}_i}(0,L):=\le\{\,x
\in\,W^{2,2}(0,L)\,:\,\mathcal{N}_i x(0)=\mathcal{N}_i
x(L)=0\,\r\},\ \ \ \ \ i=1,2.\] By interpolation we have that for
any $0\leq r\leq s\leq 1/2$ and $t>0$ the semigroups $e^{tA}$ and
$e^{t B}$ map $W^{r,2}(0,L)$ into $W^{s,2}(0,L)$\footnote{For any
$s>0$, $W^{s,2}(0,L)$ denotes the set of functions  $x \in\,H$ such that 
\[[x]_{s,2}:=\int_{[0,L]^2}\frac{|x(\xi)-x(\eta)|^2}{|\xi-\eta|^{2s+1}}\,d\xi\,d\eta<\infty.\]
$W^{s,2}(0,L)$ is endowed with the norm $|x|_{s,2}:=|x|_H+[x]_{s,2}.$} and
\begin{equation}
\label{wit21} |e^{tA} x|_{s,2}+|e^{tB} x|_{s,2}\leq c_{r,s}
(t\wedge 1)^{-\frac{s-r}2} e^{\gamma_{r,s} t}\,|x|_{r,2},
\end{equation}
for some constants $c_{r,s}\geq 1$ and $\gamma_{r,s} \in\,\reals$.

\medskip

In what follows, we shall assume that the operator $B$ arising in
the fast motion equation fulfills the following condition.

\begin{Hypothesis}
\label{H1} There exists a complete orthonormal basis  $\{e_k\}_{k
\in\,\nat}$ in $H$ and a sequence $\{\a_k\}_{k \in\,\nat}$ such
that $B e_k=-\a_k e_k$ and
\begin{equation}
\label{av3bis} \la:=\inf_{k \in\,\nat}\,\a_k>0.
\end{equation}

\end{Hypothesis}

 From \eqref{av3bis} it  immediately follows
\begin{equation}
\label{av3} \|e^{t B}\|_{\mathcal{L}(H)}\leq e^{- \la t},\ \ \ \ \
t\geq 0.
\end{equation}

\begin{Lemma}
There exists   $\gamma <1$ such that
\begin{equation}
\label{h11} \sum_{k=1}^\infty e^{-t\a_k} \leq c\,(t\wedge
1)^{-\gamma}\,e^{-\la t},\ \ \ \ t\geq 0. \end{equation}
In particular
\begin{equation}
\label{wit9} \|e^{t B}\|_2\leq c\,(t\wedge 1)^{-\frac \gamma
2}\,e^{-\la t},\ \ \ \ t\geq 0.
\end{equation}
\end{Lemma}

\begin{proof}
For any $\gamma>0$, there exists some $c_\gamma>0$ such that
\[\sum_{k=1}^\infty e^{-\a_k
t}\leq c_\gamma t^{-\gamma}\,\sum_{k=1}^\infty \a_k^{-\gamma}.\]
Now, for any second order uniformly elliptic operator on the interval $[0,L]$ having continuous coefficients, it holds $\a_k\sim k^2$. Hence, if we assume $\gamma
>1/2$ and take $t \in\,(0,1]$, we have that \eqref{h11} is satisfied. In the case $t>1$, thanks to
\eqref{av3} we have
\[\sum_{k=1}^\infty e^{-\a_k
t}=\sum_{k=1}^\infty |e^{(t-1)B} e^B e_k|_H\leq
c\,e^{-(t-1)\la}\sum_{k=1}^\infty | e^B e_k|_H\leq c\,e^{-\la
t},\] so that \eqref{h11} follows in the general case.

\end{proof}

According to  \eqref{wit9},  there
exists some $\d>0$ such that
\begin{equation}
\label{h13} \int_0^t s^{-\d} \,\|e^{s B}\|_2^2\,ds<\infty,\ \ \ \
\ t\geq 0.
\end{equation}
As known (for a prof see e.g. \cite{dpz1}), this implies that the
so-called {\em stochastic convolution} \[w^B(t):=\int_0^t
e^{(t-s)B}\,dw(s),\ \ \ \ t\geq 0,\] is a $p$ integrable
$H$-valued process, for any $p\geq 1$, having continuous
trajectories. Moreover, as a consequence of the dissipativity
assumption \eqref{av3bis},  for any $p\geq 1$
\begin{equation}
\label{eq17} \sup_{t\geq 0}\,\E\,|w^B(t)|^p_{H}=:c_{p}<\infty.
\end{equation}

\medskip

Concerning the reaction coefficient $f$ in the slow motion
equation, we assume what follows.

\begin{Hypothesis}
\label{H2} The mapping $f:[0,L]\times \mathbb{R}^2\to \reals$ is
measurable and $f(\xi,\cdot):\reals^2\to \reals$ is continuously
differentiable, for almost all $\xi \in\,[0,L]$, with uniformly bounded derivatives.
\end{Hypothesis}

Concerning the reaction coefficient $g$ in the fast motion
equation, we assume the following conditions.
\begin{Hypothesis}
\label{H3}
\begin{enumerate}
\item The mapping $g:[0,L]\times
\mathbb{R}^2\to \reals$ is measurable.

\item For each fixed $\sigma_2 \in\,\reals$ and almost all $\xi \in\,[0,L]$, the
mapping $g(\xi,\cdot,\si_2):\reals\to \reals$ is of class $C^1$, with uniformly bounded derivatives.

\item For each fixed $\si_1 \in\,\reals$ and almost all $\xi \in\,[0,L]$, the
mapping $g(\xi,\si_1,\cdot):\reals\to \reals$ is of class $C^3$, with uniformly bounded derivatives.
Moreover,

\begin{equation}
\label{av4} \sup_{\substack{\xi \in\,[0,L]\\\si
\in\,\mathbb{R}^2}}|\frac{\partial g}{\partial
\si_2}(\xi,\si)|=:L_{g}<\la,
\end{equation}
where $\la$ is the positive constant introduced in \eqref{av3bis}.

\end{enumerate}
\end{Hypothesis}

\medskip

In what follows we shall denote by $F$ and $G$  the Nemytskii
operators associated respectively with $f$ and $g$, that is
\[F(x,y)(\xi)=f(\xi,x(\xi),y(\xi)),\ \ \ \ \
G(x,y)(\xi)=g(\xi,x(\xi),y(\xi)),\] for any $\xi \in\,[0,L]$ and
$x,y \in\,H$. Due to the boundedness assumptions on their derivatives, functions $f$ and $g$ are Lipschitz-continuous and hence the
mappings $F,G:H\times H\to H$ are Lipschitz-continuous.

Concerning their regularity properties, for any fixed $y \in\,H$
the mappings $F(\cdot,y)$ and $G(\cdot,y)$ are once G\^ateaux
differentiable in $H$ with
\[D_xF(x,y)z=\frac{\partial f}{\partial
\si_1}(\cdot,x,y)z,\ \ \ \ D_xG(x,y)z=\frac{\partial g}{\partial
\si_1}(\cdot,x,y)z.\]  Moreover, for any fixed $x \in\,H$, the
mapping $F(x,\cdot):H\to H$ is once G\^ateaux
differentiable and the mapping $G(x,\cdot):H\to H$ is three times G\^ateaux
differentiable, with
\[D_yF(x,y)z=\frac{\partial f}{\partial
\si_2}(\cdot,x,y)z,\]
and
\[D_y^jG(x,y)(z_i,\ldots,z_j)=\frac{\partial g^j}{\partial
\si_2^j}(\cdot,x,y)z_1\cdots z_j,\ \ \ \ j=1,2,3.\] Notice that if
$h \in\, L^\infty(0,L)$, then for any fixed $x,y \in\,H$ the
mappings $\le<F(\cdot,y),h\r>_H$ and $\le<F(x,\cdot),h\r>_H$ are
both
 Fr\'echet differentiable and
\begin{equation}
\label{stimaderf} \sup_{x, y
\in\,H}\,|D\le<F(\cdot,y),h\r>_H(x)|_H\leq L_f\,|h|_H,\ \ \ \ \
\sup_{x, y \in\,H}\,|D\le<F(x,\cdot),h\r>_H(y)|_H\leq L_f\,|h|_H,
\end{equation}
where $L_f$ is the Lipschitz constant of $f$.

\section{Preliminary results on the  fast motion equation}
\label{sec3}

As \eqref{h13} holds and the mappings $F,G:H\times H\to H$ are both 
Lipschitz-continuous,
for any $\e>0$ and  any initial conditions $x, y \in\,H$ system
\eqref{sistema} admits a unique mild solution $(u^\e,v^\e)
\in\,\mathcal{C}_{T,p}\times \mathcal{C}_{T,p}$, with $p\geq 1$
and $T>0$ (for a proof see e.g. \cite[Theorem 7.6]{dpz1}). This
means that there exist two unique processes $u^\e$ and $v^\e$,
both in $\mathcal{C}_{T,p}$, such that for any $t \in\,[0,T]$
\begin{equation}
\label{mild} u^\e(t)=e^{t A}x+\int_0^t e^{(t-s)A}
F(u^\e(s),v^\e(s))\,ds\end{equation}
 and
\[v^\e(t)=e^{t B/\e}y+\frac 1\e \int_0^t e^{(t-s)B/\e} G(u^\e(s),v^\e(s))\,ds+\frac 1{\sqrt{\e}}
\int_0^t e^{(t-s)B/\e}\,dw(s).\]

\subsection{The  fast motion equation}
\label{3.0}

Now, for any fixed $x \in\,H$ we consider the problem
\begin{equation}
\label{av5} \le\{
\begin{array}{l}
\ds{\frac{\partial v}{\partial t}(t,\xi)=\mathcal{B}
v(t,\xi)+g(\xi,x(\xi),v(t,\xi))+\frac{\partial w}{\partial
t}(t,\xi), \ \ \ \ \ t\geq 0,\ \ \ \xi \in\,[0,L],}\\
\vs \ds{v(0,\xi)=y(\xi),\ \ \ \xi \in\,[0,L],\ \ \ \ \
 \mathcal{N}_2v\,(t,\xi)=0,\ \ \ \ t\geq 0,\ \ \ \xi \in\,\{0,L\}.}
 \end{array}\r.
\end{equation}
By arguing as above, for any fixed slow component $x \in\,H$ and any
initial datum $y \in\,H$, equation \eqref{av5} admits a unique
mild solution in $\mathcal{C}_{T,p}$, which will be denoted by
$v^{x,y}(t)$.

Moreover, as proved for example in \cite[Proposition 8.2.2]{cerrai}, there exists $\theta>0$ such that for any $t_0>0$
and $p\geq 1$
\begin{equation}
\label{suppo} \sup_{t\geq t_0}\E\,|v^{x,y}(t)|_{C^\theta([0,L])}^p<\infty.
\end{equation}

\begin{Lemma}
\label{L3.1}
 Under Hypotheses \ref{H1} and \ref{H3}, for any
$p\geq 1$ and $x,y \in\,H$
\begin{equation}
\label{wit1bis} \E\,|v^{x,y}(t)|_H^p\leq c_{p}\,\le(e^{-\d p
t}|y|_H^p+|x|_H^p+1\r),\ \ \ \ t\geq 0,
\end{equation}
where $\d:=(\la-L_{g})/2$.
\end{Lemma}

\begin{proof}
If we set $\rho(t):=v^{x,y}(t)-w^B(t)$,  thanks to \eqref{av3} and
\eqref{av4} and to the Lipschitz-continuity of $G$, we have
\[\begin{array}{l}
\ds{\frac 12 \frac{d}{dt}|\rho(t)|_H^2}\\
\vs \ds{=\le<B
\rho(t),\rho(t)\r>_H+\le<G(x,\rho(t)+w^B(t))-G(x,w^B(t)),\rho(t)\r>_H+\le<G(x,w^B(t)),\rho(t)\r>_H}\\
\vs \ds{\leq
-(\la-L_{g})\,|\rho(t)|_H^2+c\,\le(|w^B(t)|_H+|x|_H+1\r)\,|\rho(t)|_H}\\
\vs \ds{\leq -\frac{\la-L_{g}}2\,|\rho(t)|_H^2+
c\,\le(|w^B(t)|_H^2+|x|^2_H+1\r)}
\end{array}\]
and, by comparison, it easily follows
\begin{equation}
\label{wit1} |v^{x,y}(t)|_H\leq |\rho(t)|_H+|w^B(t)|_H\leq
c\,\le(e^{-\frac{\la-L_{g}}2\,t}\,|y|_H+\sup_{s\geq
0}\,|w^B(s)|_H+|x|_H+1\r).\end{equation} In particular, if we set
$\d:=(\la-L_{g})/2$, as a consequence of estimate  \eqref{eq17}
we obtain \eqref{wit1bis}.

\end{proof}

 Since we are assuming that for each fixed $\si_1
\in\,\reals$ and almost all $\xi \in\,[0,L]$   the mapping
$g(\xi,\si_1,\cdot):\reals\to \reals$ is of class $C^3$, with
uniformly bounded derivatives, for any $T>0$ and $p\geq 1$ and for
any fixed slow variable $x \in\,H$ the mapping
\begin{equation}
\label{av6}
y \in\,H\mapsto v^{x,y} \in\,\mathcal{H}_{T,p},
 \end{equation}
is three times continuously differentiable (for a proof and all
details see e.g. \cite[Theorem 4.2.4]{cerrai}).

The first order derivative $D_y v^{x,y}(t)h$, at the point $y
\in\,H$ and along the direction $h \in\,H$, is the solution of the
first variation equation
\[\le\{
\begin{array}{l}
\ds{\frac{\partial z}{\partial
t}(t,\xi)=\mathcal{B}z(t,\xi)+\frac{\partial g}{\partial
\si_2}(\xi,x(\xi),y(\xi))z(t,\xi),}\\
\vs \ds{z(0)=h,\ \ \ \ \ \mathcal{N}_2 z\,(t,\xi)=0,\ \ \ \ \xi
\in\,\{0,L\}.}
\end{array}
\r.\]  Hence, thanks to \eqref{av3} and \eqref{av4}, it is
immediate to check that for any $t\geq 0$
\begin{equation}
\label{av7} \sup_{x,y \in\,H}\,|D_y v^{x,y}(t)h|_H\leq e^{-\d
t}\,|h|_H,\ \ \ \ \ \mathbb{P}-\mbox{a.s.}
\end{equation}
where, as in the previous lemma, $\d=(\la-L_{g})/2$. Moreover, as
shown in \cite[Lemma 4.2.2]{cerrai}, for any $1\leq r\leq p\leq
\infty$ and $h \in\,L^r(0,L)$ we have that $D_y v^{x,y}(t)h
\in\,L^p(0,L)$, $\mathbb{P}$-a.s. for $t>0$, and
\[\sup_{y
\in\,H}|D_y v^{x,y}(t)h|_{L^p}\leq \mu_{r,p}(t)\,
t^{-\frac{p-r}{2rp}}\,|h|_{L^r},\ \ \ \ \ \mathbb{P}-\mbox{a.s.}
\]
for a continuous increasing function $\mu_{r,p}$ which is
independent of $x \in\,H$.

Concerning the second and the third order derivatives, they are
respectively solutions of the second and of the third variation
equations. As proved in \cite[Proposition 4.2.6]{cerrai}, for any
$h_1,h_2,h_3 \in\,H$ and $p\geq 1$ both $D_y^2
v^{x,y}(t)(h_1,h_2)$ and $D_y^3 v^{x,y}(t)(h_1,h_2,h_3)$ belong to
$L^p(0,L)$, $\mathbb{P}$-a.s. for any $t\geq 0$, and
\begin{equation}
\label{av8} \sup_{y \in\,H}|D_y^j
v^{x,y}(t)(h_1,\ldots,h_j)|_{L^p}\leq
\nu^j_{r,p}(t)\prod_{i=1}^j|h_i|_H,\ \ \ \ \
\mathbb{P}-\mbox{a.s.}, \end{equation} for $j=2,3$. It is
important to notice that, as for  $\mu_{r,p}$,  due to the
boundedness assumption on the derivatives of the reaction term
$g$, all $\nu^j_{r,p}$ are
 continuous increasing functions  independent of $x\in\,H$.

 \medskip

 We conclude this subsection by proving  the smooth dependence of
the solution $v^{x,y}(t)$ of equation \eqref{av5} on the frozen
slow component $x \in\,H$. In the space $\mathcal{H}_{T,2}$ we
introduce the equivalent norm
\[\||u\||:=\sup_{t \in\,[0,T]}\,e^{-\a t}\,\E\,|u(t)|_H^2,\]
for some $\a>0$. Moreover, for any $x \in\,H$ and $v
\in\,\mathcal{H}_{T,2}$ we define
\[
\mathcal{F}(x,v)(t):= e^{t B}y+\int_0^t e^{(t-s)
B}G(x,v(s))\,ds+w^B(t), \ \ \ t \in\,[0,T].\] If $\a$ is chosen
large enough, the mapping $\mathcal{F}(x,\cdot)$ is a contraction in
the space $\mathcal{H}_{T,2}$, endowed with the norm defined
above.

It is easy to show that for all $v \in\,\mathcal{H}_{T,2}$, the
mapping $x \in\,H\mapsto \mathcal{F}(x,v) \in\,\mathcal{H}_{T,2}$
is Fr\'echet differentiable  and the derivative is continuous.
Furthermore, for all $x \in\,H$ the mapping $v
\in\,\mathcal{H}_{T,2}\mapsto \mathcal{F}(x,v)
\in\,\mathcal{H}_{T,2}$ is G\^ateaux differentiable and the
derivative is continuous. Hence, by using the generalized theorem
on contractions depending on a parameter given in
\cite[Proposition C.0.3]{cerrai}, we have that the solution
$v^{x,y}$ of equation \eqref{av5}, which is the fixed point  of
the mapping $\mathcal{F}(x,\cdot)$, is differentiable with respect
to $x \in\,H$ and the derivative along the direction $h \in\,H$
satisfies the following equation
\[\frac{d\rho}{dt}(t)=[B +G_y(x,v^{x,y}(t))]\,\rho(t)+G_x(x,v^{x,y}(t))h,\ \ \ \ \ \
\rho(0)=0.\] According to Hypothesis \ref{H3}, we have
\[\begin{array}{l}
\ds{\frac 12 \frac{d}{dt}|\rho(t)|_H^2=\le<[B
+G_y(x,v^{x,y}(t))]\rho(t),\rho(t)\r>_H+\le<G_x(x,v^{x,y}(t))h,\rho(t)\r>_H}\\
\vs \ds{\leq
-\frac{\la-L_{g}}2\,|\rho(t)|_H^2+|G_x(x,v^{x,y}(t))|^2_{\mathcal{L}(H)}|h|_H^2,}
\end{array}\]
so  that, due to the boundedness of $G_x$,
\begin{equation}
\label{finis} \sup_{x, y \in\,H}\,|D_x v^{x,y}(t)h|_H\leq c\,
e^{-\frac{\la-L_{g}}2\,t}|h|_H,\ \ \ \ \ \mathbb{P}-\text{a.s.}
\end{equation}

\subsection{The  fast transition semigroup}
\label{3.1}

For any fixed $x \in\,H$, we denote by  $P^x_t$, $t\geq 0$,  the
transition semigroup
 associated with the fast equation \eqref{av5} with frozen slow component $x$. For any $\varphi
\in\,B_b(H)$ and $t\geq0$, it is defined by
\[P_t^x \varphi(y)=\E\,\varphi(v^{x,y}(t)),\ \ \ \ \ y \in\,H.\]
 As the mapping introduced in
\eqref{av6} is differentiable and \eqref{av7} holds, it is
immediate to check that $P^x_t$ is a Feller contraction semigroup and  maps $C_b(H)$ into itself.

Thanks to estimate \eqref{wit1bis} and to \eqref{fi}, the
semigroup $P^x_t$ is well defined on $\text{Lip}(H)$ and for any
$\varphi \in\,\text{Lip}(H)$, $x,y \in\,H$ and $t\geq 0$
\begin{equation}
\label{lippt} |P^x_t \varphi(y)|\leq [\varphi]_{\text{\tiny
Lip}}\,\E\,|v^{x,y}(t)|_H+|\varphi(0)|\leq
c\,[\varphi]_{\text{\tiny Lip}}\,(1+|x|_H+|y|_H)+|\varphi(0)|.
\end{equation}
Furthermore, $P^x_t$ maps $\text{Lip}(H)$ into itself and
according to \eqref{av7}
 \begin{equation}
 \label{lipptbis}
 [P^x_t\varphi]_{\text{\tiny Lip}}\leq e^{-\d t}[\varphi]_{\text{\tiny
Lip}},\ \ \ \ \ t\geq 0.
\end{equation}

As known, the semigroup $P^x_t$ is not strongly continuous on
$C_b(H)$, in general. Nevertheless, it is {\em weakly continuous}
on $C_b(H)$ (for a definition and all details we refer to
\cite[Appendix B]{cerrai}). For any $\la>0$  and $\varphi
\in\,C_b(H)$, we set
\[F^x(\la)\varphi(y):=\int_0^\infty e^{-\la t}P^x_t\varphi(y)\,dt,\ \ \
\ \ x,y \in\,H.\] As proved in \cite[Proposition B.1.3 and
Proposition B.1.4]{cerrai}, since  $P^x_t$ is a weakly continuous
semigroup, for any $\la>0$ and $x \in\,H$ the linear operator
$F^x(\la)$ is bounded from $C_b(H)$ into itself and there exists a
unique closed linear operator $L^x:D(L)\subseteq C_b(H)\to C_b(H)$
such that
\[F^x(\la)=R(\la,L^x)\ \ \ \ \ \ \la>0.\]
Such an operator is, by definition, the  {\em infinitesimal weak
generator} of $P^x_t$.

It is important to stress that, thanks to \eqref{lippt} and
\eqref{lipptbis}, the operator $F^x(\la)$ is also well defined
from $\text{Lip}(H)$ into itself.

\medskip

Concerning the regularity properties of $P^x_t$,
 as the mapping \eqref{av6} is three times
continuously differentiable, by differentiating under the sign of
expectation, for any $t\geq 0$ and $k\leq 3$ we get
\[\varphi \in\,\text{Lip}^k(H)\Longrightarrow P^x_t\varphi
\in\,\text{Lip}^k(H),\] and thanks to estimates \eqref{av7}, for
$k=1$, and \eqref{av8}, for $k=2,3$,
\[\sup_{x \in\,H}[P^x_t\varphi]_k\leq  c_k(t)\,\sum_{1\leq h\leq k}[\varphi]_h,\ \ \ \ \ t \geq 0,\]
where $c_k(t)$ is some continuous increasing function. Moreover,
the semigroup $P^x_t$ has a smoothing effect. Actually, as proved
in \cite[Theorem 4.4.5]{cerrai}, for any $t>0$
\[\varphi \in\,B_b(H)\Longrightarrow P^x_t\varphi \in\,C^3_b(H),\]
and for any $0\leq i\leq j\leq 3$
\begin{equation}
\label{eq13}
 \sup_{x \in\,H}\|D^j (P^x_t\varphi)\|_0\leq
c\,(t\wedge 1)^{-\frac{j-i}2}\,\|\varphi\|_i,\ \ \ \ \ t>0.
\end{equation}

By adapting the arguments used in \cite[Theorem 4.4.5]{cerrai}, it
is possible to prove that if $\varphi \in\,\text{Lip}(H)$, then
$P^x_t\varphi$ is three times continuously differentiable, for any
$t>0$. Moreover, the following estimates for the derivatives of
$P^x_t\varphi$ hold
\begin{equation}
\label{deriprima} [P^x_t\varphi]_1=\sup_{y \in\,H}\,|D
P^x_t\varphi(y)|_H\leq e^{-\d t}\,[\varphi]_{\text{\tiny Lip}},
\end{equation}
and for $j=2,3$
\begin{equation}
\label{derisucc} |D^j P^x_t\varphi(y)|_{\mathcal{L}^j(H)}\leq
c\,(t\wedge 1)^{-\frac{j-1}2}\,\le([\varphi]_{\text{\tiny
Lip}}(1+|x|_H+|y|_H)+|\varphi(0)|\r).
\end{equation}

Moreover, by adapting the proof of \cite[Theorem 5.2.4]{cerrai}, which is
given for bounded functions, to the case of general Lipschitz-continuous functions, it is possible to prove the following
crucial fact.

\begin{Theorem}
\label{av11} Under Hypotheses \ref{H1} and \ref{H3}, the operator
$D^2(P^x_t  \varphi)(y)$ belongs to $\mathcal{L}_1(H)$, for any
fixed $x,y \in\,H$, $t>0$ and $\varphi \in\,\text{{\em Lip}}(H)$.
Besides, the mapping
\[(t,y) \in\,(0,\infty)\times H\mapsto
\mbox{{\em Tr}}\,[D^2(P^x_t\varphi)(y)] \in\,\reals,\] is
continuous and
\begin{equation}
\label{eq14} \le|\mbox{{\em Tr}}\,[D^2(P^x_t\varphi)(y)]\r|\leq
c_\gamma\,(t\wedge
1)^{-\frac{1+\gamma}2}\le(\,[\varphi]_{\text{{\em \tiny
Lip}}}(1+|x|_H+|y|_H)+|\varphi(0)|\r),
\end{equation}
where $\gamma$ is the constant introduced in \eqref{h11}.
\end{Theorem}

\begin{Remark}
{\em   Even if the semigroup $P^x_t$ has a smoothing effect, the
proof of the validity of the trace-class property for the operator $D^2(P^x_t
\varphi)(y)$ is far from being trivial. Actually, it is based on
the two following facts. First (see \cite[Lemma 5.2.1]{cerrai}),
if $\{e_k\}_{k \in\,\nat}$ is the orthonormal basis introduced in
Hypothesis \ref{H1} and if $\gamma <1$ is the constant
introduced in \eqref{h11}, then it holds
\[\sup_{y \in\,H}\,\sum_{k=1}^\infty \int_0^t|D_y
v^{x,y}(s)e_k|_H^2\,ds\leq c(t)\,t^{1-\gamma},\ \ \ \
\mathbb{P}-\mbox{a.s.},\] for some continuous increasing function
$c(t)$ independent of $x \in\,H$. Secondly (see \cite[Lemma
5.2.2]{cerrai}), there exists some continuous increasing function
$c(t)$ such that for any $N \in\,\mathcal{L}(H)$ and $x \in\,H$
\[\sup_{y \in\,H}\sum_{k=1}^\infty \int_0^t|D^2_y
v^{x,y}(s)(e_k,N e_k)|_H^2\,ds\leq c(t)\,\|N\|,\ \ \ \
\mathbb{P}-\mbox{a.s.}\] 
It is important to stress that both the estimate for the first
derivative and the estimate for the second derivative are a
consequence of \eqref{wit9} and \eqref{h11}.

 }
\end{Remark}

\subsection{The asymptotic behavior of the  fast equation}
\label{3.2}

We describe here the asymptotic behavior of the semigroup $P^x_t$.
Namely, we show that, for any fixed $x \in\,H$,  it admits a
unique invariant measure $\mu^x$ which is explicitly given  and we
describe the convergence of $P^x_t$ to equilibrium. Most of these
results are basically known in the literature, but we shortly
recall them for the reader convenience.

\medskip

According to \eqref{wit9}, the self-adjoint operator
\[\int_0^\infty e^{2 s B}\,ds=\frac 12\,(-B)^{-1}\]
is well defined in $\mathcal{L}_1(H)$, so that the Gaussian
measure $\mathcal{N}(0,(-B)^{-1}/2)$ of zero mean and covariance
operator $(-B)^{-1}/2$ is well defined on $(H,\mathcal{B}(H))$.

For any $x, y \in\,H$ we define
\[U(x,y):=\int_0^1 \le<G(x,\theta y),y\r>_H\,d\theta=
\int_0^L\int_0^{y(\xi)} g(\xi,x(\xi),s)\,ds\,d\xi.\] Due to the
Lipschitz-continuity of $g(\xi,\cdot)$ (see Hypothesis \ref{H3}),
for any $x, y \in\,H$ we have
\begin{equation}
\label{G} |G(x,y)|_H\leq
L_{g}\,|y|_H+c\,|x|_H+|G(0,0)|_H,\end{equation}
 so that
for any $\eta>0$ we can fix a constant $c_\eta\geq 0$ such that
\[|U(x,y)|\leq
\frac{L_{g}+\eta}2\,|y|_H^2+c_\eta\,(1+|x|_H^2).\] As $\eta>0$ can be chosen as small as we wish, thanks to
\eqref{av4} this implies that for any fixed $x \in\,H$ the
mapping
\[y \in\,H\mapsto \exp 2 U(x,y) \in\,\reals\]
is integrable with respect to the Gaussian measure
$\mathcal{N}(0,(-B)^{-1}/2)$ and
\[Z(x):=\int_H \exp 2 U(x,y)\,\mathcal{N}(0,(-B)^{-1}/2)\,dy
\in\,(0,\infty),\ \ \ \ \ x \in\,H.\] This means that for each
fixed $x \in\,H$ the measure
\begin{equation}
\label{3.17}
\mu^x(dy):=\frac 1{Z(x)}\,\exp 2
U(x,y)\,\mathcal{N}(0,(-B)^{-1}/2)\,(dy)
\end{equation}
is well defined on $(H,\mathcal{B}(H))$.

Now, it is immediate to check  that the mapping $U(x,\cdot):H\to
\reals$ is differentiable and
\begin{equation}
\label{dery} U_y(x,y)=G(x,y),\ \ \ \ x,y \in\,H.
\end{equation}
Therefore, as well known from  the existing  literature, the
measure $\mu^x$ defined in \eqref{3.17} is invariant for equation \eqref{av5}.

Because of the way the measure $\mu^x$ has been constructed, we
immediately have that it has all moments finite. In particular,
for any $x \in\,H$ we have
\begin{equation}
\label{lip} \text{Lip}(H)\subset L^p(H,\mu^x),\ \ \ \ p\geq
1.\end{equation} In the next lemma we show how the moments of
$\mu^x$ can be estimated in terms of the slow variable $x$.

\begin{Lemma}
\label{tight} Under Hypotheses \ref{H1} and \ref{H3}, for any $x
\in\,H$ and $p\geq 1$
\begin{equation}
\label{wit22} \int_H |z|^p_H\,\mu^x(dz)\leq c\,\le(1+|x|^p_H\r).
\end{equation}
\end{Lemma}

\begin{proof}
By using   the invariance of $\mu^x$,  thanks to estimate
\eqref{wit1bis} for any $p\geq 1$ and $t\geq 0$ we have
\[\begin{array}{l}
\ds{\int_H|z|_H^p\,\mu^x(dz)=\int_H
P^x_t|z|_H^p\,\mu^x(dz)=\int_H\E\,|v^{x,z}(t)|_H^p\,\mu^x(dz)}\\
\vs \ds{\leq c\,e^{-\d p
t}\int_H|z|_H^p\,\mu^x(dz)+c\,(1+|x|_H^p)}
\end{array}\]
Then, if we take $t=t_0$  such that $c\,e^{-\d p t_0}<1$, we have
\eqref{wit22}.

\end{proof}

Once we have the explicit invariant measure $\mu^x$,  we show that
it is unique and we describe its convergence to equilibrium.

\begin{Theorem}
\label{ergo} Under Hypotheses \ref{H1} and  \ref{H3}, for any
fixed $x \in\,H$ equation \eqref{av5} admits a unique ergodic
invariant measure $\mu^x$, which is strongly mixing and such that
for any $\varphi \in\,B_b(H)$ and $x, y \in\,H$
\begin{equation}
\label{mis5} \le|P^x_t\varphi(y)-\int_H
\varphi(z)\,\mu^x(dz)\r|\leq c\,\le(1+|x|_H+|y|_H\r)\,e^{-\d
t}(t\wedge 1)^{-\frac 12}\,\|\varphi\|_0, \end{equation} where
$\d:=(\la-L_{g})/2$.

\end{Theorem}

\begin{proof}
We fix $y,z \in\,H$ and set $\rho(t):=v^{x,y}(t)-v^{x,z}(t)$. We
have
\[\frac 12\frac{d}{dt}|\rho(t)|_H^2=\le<B \rho(t),\rho(t)\r>_H+\le<G(x,v^{x,y}(t))-G(x,v^{x,z}(t)),
\rho(t)\r>_H,\] and then, according to \eqref{av3} and
\eqref{av4}, we easily get
\[|v^{x,y}(t)-v^{x,z}(t)|_H^2=|\rho(t)|_H^2\leq e^{-2(\la-L_{g})t}|y-z|_H^2,\ \ \ \ \mathbb{P}-\text{a.s}.\]
This means that for any $\varphi \in\,\text{Lip}(H)$
\[ |P^x_t\varphi(y)-P^x_t\varphi(z)|\leq [\varphi]_{\text{\tiny Lip}}
\,\E\,|v^{x,y}(t)-v^{x,z}(t)|_H\leq [\varphi]_{\text{\tiny Lip}}
\,e^{-(\la-L_{g})t}|y-z|_H,\ \ \ \ t\geq 0.
\]
 Hence, if $\varphi \in\,B_b(H)$, due to the
semigroup law and to estimate \eqref{eq13} (with $j=1$ and $i=0$),
 for any $t>0$ we obtain
\begin{equation}
\label{wit8}
\begin{array}{l} \ds{|P^x_t\varphi(y)-P^x_t\varphi(z)|
\leq
[P^x_{t/2}\varphi]_1\,e^{-\d\,t}\,|y-z|_H
\leq c\,\|\varphi\|_0 \,(t\wedge 1)^{-\frac 12}\,e^{-\d\,
t}\,|y-z|_H,}
\end{array}\end{equation}
where  $\d:=(\la-L_{g})/2$. In particular,
\[\lim_{t\to \infty} P^x_t\varphi(y)-P^x_t\varphi(z)=0,\]
so that the invariant measure $\mu^x$ is unique and strongly
mixing.

Now, due to the invariance of $\mu^x$, if $\varphi \in\,B_b(H)$
from \eqref{wit8}  we have
\[\begin{array}{l}
\ds{\le|P^x_t\varphi(y)-\int_H
\varphi(z)\,\mu^x(dz)\r|=\le|\int_H\le[P^x_t\varphi(y)-P^x_t\varphi(z)\r]\,\mu^x(dz)\r|}\\
\vs \ds{\leq c\,\|\varphi\|_0 \,e^{-\d\, t}\,(t\wedge
1)^{-\frac 12}\int_H|y-z|_H\,\mu^x(dz)}\\
\vs \ds{\leq c\,\|\varphi\|_0\,e^{-\d\, t}\,(t\wedge 1)^{-\frac
12}\le(|y|_H+\int_H|z|_H\,\mu^x(dz)\r).}
\end{array}\]
and then, thanks to \eqref{wit22} (with $p=1$), we obtain
\eqref{mis5}.

\end{proof}

\begin{Remark}
{\em From the proof of estimate \eqref{mis5}, we immediately see
that if $\varphi \in\,\text{Lip}(H)$, then for any $x,y \in\,H$
\begin{equation}
\label{mis5bis} \le|P^x_t\varphi(y)-\int_H
\varphi(z)\,\mu^x(dz)\r|\leq c\,\le(1+|x|_H+|y|_H\r)\,e^{-\d
t}\,[\varphi]_{\text{\tiny Lip}} ,
\end{equation} where $\d:=(\la-L_{g})/2$.}
\end{Remark}

\subsection{The Kolmogorov equation associated with the  fast
equation} \label{3.3}

For any frozen slow component $x \in\,H$, the Kolmogorov operator
associated with equation \eqref{av5} is given by the following
second order differential operator
\[\mathcal{L}^x \varphi(y)=\frac12 \mbox{Tr}\,[D^2
\varphi(y)]+\le<By+G(x,y),D\varphi(y)\r>_H,\ \ \ \ \ y
\in\,D(B).\] The operator $\mathcal{L}^x$ is defined for functions
$\varphi:H\to \reals$  which are twice continuously
differentiable, such that the operator $D^2 \varphi(y)$ is in
$\mathcal{L}_1(H)$, for all $y \in\,H$, and the mapping \[y
\in\,H\mapsto \mbox{Tr}\,D^2 \varphi(y) \in\,\reals,\] is
continuous. In what follows it will be important to study the
solvability of the elliptic equation associated with the infinite
dimensional operator $\mathcal{L}^x$
\begin{equation}
\label{av12} \la \varphi(y)-\mathcal{L}^x \varphi(y)=\psi(y),\ \ \
\ \ y \in\,D(B),
\end{equation}
for any fixed $x \in\,H$, $\la>0$ and $\psi:H\to \reals$ regular
enough. To this purpose we recall the notion of {\em strict}
solution for the elliptic problem \eqref{av12}.
\begin{Definition}
\label{def3} A function $\varphi$ is a strict solution of problem
\eqref{av12} if
\begin{enumerate}
\item $\varphi$ belongs to $D(\mathcal{L}^x)$, that is $\varphi:H\to \reals$
is twice continuously differentiable, the operator $D^2\varphi(y)
\in\,\mathcal{L}_1(H)$, for any $y \in\,H$, and the mapping
$y\mapsto \mbox{{\em Tr}}\,D^2\varphi(y)$ is continuous on $H$
with values in $\reals$;
\item $\varphi(y)$ satisfies \eqref{av12}, for any $y \in\,D(B)$.
\end{enumerate}
\end{Definition}

 In the next theorem we see how it is possible to get the
existence of a strict solution of problem \eqref{av12} in terms of
the Laplace transform of the semigroup $P^x_t$ (see subsection
\ref{3.1} for the definition and  \cite[Theorem 5.4.3]{cerrai} for
the proof).

\begin{Theorem}
\label{av14} Fix any $x \in\,H$ and $\la>0$. Then under Hypotheses
\ref{H1} and \ref{H3}, for any $\psi \in\,\text{{\em Lip}}(H)$ the
function
\[y \in\,H\mapsto \varphi(x,y):=\int_0^\infty e^{-\la
t}P^x_t\psi(y)\,dt \in\,\reals,\] is a strict solution of problem
\eqref{av12}.
\end{Theorem}

\begin{Remark}
{\em  A detailed proof of the theorem above can be found in
\cite[Theorem 5.4.3]{cerrai} in the case $\psi \in\,C^1_b(H)$. The case of $\psi \in\,\text{Lip}(H)$ is analogous:
we have to
prove that  for any $\psi \in\,\text{{ Lip}} (H)$ the function
$R(\la,L^x)\psi$ is a strict solution. To this purpose, by using
\eqref{deriprima} and \eqref{derisucc}, we have that
$R(\la,L^x)\psi$ is twice continuously differentiable and then,
thanks to Theorem \ref{av11} and estimate \eqref{eq14}, we have
 that $D^2[R(\la,L^x)\psi] \in\,\mathcal{L}_1(H)$ and
continuity for the trace holds. Notice that in all these results
it is crucial that $\psi \in\,\text{{ Lip}} (H)$, because  in this
case all singularities arising at $t=0$ are integrable. }
\end{Remark}

\section{A priori bounds for the solution of the system}
\label{sec4}

With the notations introduced in section \ref{sec2}, system
\eqref{sistema} can be written as \begin{equation}
\label{astratta} \le\{
\begin{array}{l}
\ds{ \frac{d u^\e}{dt}(t)=A u^\e(t)+F(u^\e(t),v^\e(t)),\ \ \ \
u^\e(0)=x,}\\
\vs \ds{ dv^\e(t)=\frac 1\e\,\le[B
v^\e(t)+G(u^\e(t),v^\e(t))\r]\,dt+\frac 1{\sqrt{\e}}\,dw(t),\ \ \
\ v^\e(0)=y.}
\end{array}\r.
\end{equation}
Our aim here is proving uniform bounds with respect to $\e>0$ for
the solutions $u^\e$ and $v^\e$ of system \eqref{astratta}.

\begin{Lemma}
\label{stime1} Under Hypotheses \ref{H1}, \ref{H2} and \ref{H3},
for any $x, y \in\,H$ and $T>0$ we have
\begin{equation}
\label{stima11} \sup_{\e>0}\,\E \sup_{t \in\,[0,T]}|u^\e(t)|_H^2\leq
c_T\le(1+|x|_H^2+|y|^2_H\r),
\end{equation}
and
\begin{equation}
\label{stima12} \sup_{\e>0}\,\sup_{t
\in\,[0,T]}\E\,|v^\e(t)|_H^2\leq c_T\le(1+|x|_H^2+|y|^2_H\r),
\end{equation}
for some constant $c_T>0$.
\end{Lemma}

\begin{proof}
We have
\[\begin{array}{l}
\ds{\frac 12 \frac{d}{dt}|u^\e(t)|_H^2=\le<A u^\e(t),
u^\e(t)\r>_H+\le<F(u^\e(t),v^\e(t))-F(0,v^\e(t)),u^\e(t)\r>_H}\\
\vs \ds{+\le<F(0,v^\e(t)),u^\e(t)\r>_H\leq
c\,|u^\e(t)|_H^2+c\,\le(1+|v^\e(t)|_H^2\r),}
\end{array}\]
so that
\begin{equation}
\label{wit10} |u^\e(t)|_H^2\leq e^{ct}\,|x|_H^2+c\,\int_0^t
e^{c(t-s)}\,\le(1+|v^\e(s)|_H^2\r)\,ds.
\end{equation}

Now, for any $\e>0$ we denote by $w^{\e,B}(t)$ the solution of the
problem
\[dz(t)=\frac 1\e Bz(t)\,dt+\frac 1{\sqrt{\e}}\, dw(t),\ \ \ \ \ \
z(0)=0.\] We have \[w^{\e,B}(t)=\frac 1{\sqrt{\e}} \int_0^t
e^{(t-s)B/\e}\,dw(s),\] and, due to \eqref{wit9},  with a simple
change of variables we get
\[\begin{array}{l}
\ds{\E\,|w^{\e, B}(t)|_H^2=\frac
1\e\,\int_0^{t}\|e^{(t-s)B/\e}\|_2^2\,ds=\int_0^{t/\e}\|e^{\rho
B}\|_2^2\,d\rho\leq c\int_0^\infty (\rho\wedge
1)^{-\gamma}e^{-2\la \rho}\,d\rho<\infty.}
\end{array}\]
This means that
\begin{equation}
\label{wit20} \sup_{\e>0}\,\sup_{t\geq 0}\,\E\,|w^{\e,
B}(t)|_H^2<\infty.
\end{equation}
Notice that the same uniform bound  is true for moments of any
order of  the $H$-norm of $w^{\e, B}(t)$.

If we set $\rho^\e(t):=v^\e(t)-w^{\e, B}(t)$, by proceeding as in
the proof of Lemma \ref{L3.1} we have
\[\frac 12 \frac d{dt}\,|\rho^\e(t)|_H^2\leq
-\frac{\la-L_{g}}{2 \e}\,|\rho^\e(t)|_H^2+\frac c
\e\,\le(1+|u^\e(t)|_H^2+|w^{\e, B}(t)|_H^2\r).\] Hence, by
comparison
\begin{equation}
\label{dea1} |\rho^\e(t)|_H^2\leq
e^{-\frac{\la-L_{g}}{\e}\,t}\,|y|_H^2+\frac c \e\,\int_0^t
e^{-\frac{\la-L_{g}}{\e}\,(t-s)}\le(1+|u^\e(s)|_H^2+|w^{\e,
B}(s)|_H^2\r)\,ds.
\end{equation}
 According to \eqref{wit10}, this
implies
\[\begin{array}{l}
\ds{|v^\e(t)|_H^2\leq 2\,|w^{\e,
B}(t)|_H^2+2\,|\rho^\e(t)|_H^2\leq 2\,|w^{\e,
B}(t)|_H^2+c_T\,\le(1+|x|_H^2+|y|_H^2\r)}\\
\vs
\ds{+\frac{c_T}\e\int_0^te^{-\frac{\la-L_{g}}{\e}\,(t-s)}\int_0^s
|v^\e(r)|_H^2\,dr\,ds+\frac
c\e\int_0^te^{-\frac{\la-L_{g}}{\e}\,(t-s)}|w^{\e,
B}(s)|_H^2\,ds}
\end{array}\]
and by taking expectation, thanks to \eqref{wit20} we have
\[\begin{array}{l}
\ds{\E\,|v^\e(t)|_H^2\leq
c_T\,\le(1+|x|_H^2+|y|_H^2\r)}\\
\vs
\ds{+\frac{c_T}\e\int_0^te^{-\frac{\la-L_{g}}{\e}\,(t-s)}\int_0^s
\E\,|v^\e(r)|_H^2\,dr\,ds+\frac
c\e\int_0^te^{-\frac{\la-L_{g}}{\e}\,s}\,ds.}
\end{array}\]
With a change of variables,  this  yields
\[\begin{array}{l}
\ds{\E\,|v^\e(t)|_H^2\leq
c_T\,\le(1+|x|_H^2+|y|_H^2\r)}\\
\vs \ds{+c_T\int_0^t\le[\int_0^{\frac{t-r}\e}
e^{-(\la-L_{g})\,\si}\,d\si\r]\E\,|v^\e(r)|_H^2\,dr+c\int_0^{\frac t \e}
e^{-(\la-L_{g})\,s}\,ds}\\
\vs \ds{\leq
c_T\,\le(1+|x|_H^2+|y|_H^2\r)+c_T\int_0^t\E\,|v^\e(r)|_H^2\,dr,}
\end{array}\]
so that
\[\sup_{t \in\,[0,T]}\,\E\,|v^\e(t)|_H^2\leq c_T\,\le(1+|x|_H^2+|y|_H^2\r),  \]
which gives \eqref{stima12}. By replacing the estimate above in
\eqref{wit10}, we immediately obtain \eqref{stima11}.
\end{proof}

\begin{Remark}
{\em In the previous lemma we have proved uniform bounds, with
respect to $\e>0$, for $\sup_{t \in\,[0,T]}\,\E\,|v^\e(t)|_H$ and
not for $\E\,\sup_{t \in\,[0,T]}\,|v^\e(t)|_H$. This is a
consequence of the fact that we can only prove the following
estimate for the second  moment of the $C([0,T];H)$-norm of the stochastic
convolution $w^{\e,B}$
\[\E\,\sup_{t\in\,[0,T]}|w^{\e, B}(t)|^2_H\leq c_{T,\d} \,\e^{\d-1},\ \ \ \ \
t \in\,[0,T],\] for any $0<\d<1/2$. Then, due to the previous
estimate, we are only able to prove that
\begin{equation}
\label{dea} \E\,\sup_{t \in\,[0,T]}\,|v^\e(t)|_H^2\leq
c_{T,\d}\le(1+|x|^2_H+|y|^2_H+\e^{\,\d-1}\r),\ \ \ \ \ \e>0.
\end{equation}}
\end{Remark}

\begin{Theorem}
\label{tightness}
 Assume that $x \in\,W^{\a,2}(0,L)$, for some
$\a>0$. Then, under Hypotheses \ref{H1}, \ref{H2} and \ref{H3},
for any $T>0$ the family of probability measures
$\{\,\mathcal{L}(u^\e)\,\}_{\e
>0}$ is tight in $C((0,T];H)\cap L^\infty(0,T;H)$.
\end{Theorem}

\begin{proof}
As known, if $\d\leq 1/4$
\begin{equation}
\label{analytic}
W^{2\d,2}(0,L)=(H,W^{2,2}_{\mathcal{N}_i}(0,L))_{\d,\infty}=\le\{\,x
\in\,H\,:\,\sup_{t \in\,(0,1]}t^{-\d}\,|e^{t
A}x-x|_H<\infty\,\r\},
\end{equation}
 with equivalence of norms.
Moreover, if $f \in\,L^2(0,T;H)$, it is possible to prove that for any $t>s$ and $\d<1/2$ 
\[\le|\int_0^t e^{(t-r)A}f(r)\,dr-\int_0^s e^{(s-r)A}
f(r)\,dr\r|_H\leq c_{T,\d}\,(t-s)^\d\|f\|_{L^2(0,T;H)}.\] Then, if $x
\in\,W^{\a,2}(0,L)$, for any $t>s$ and $\theta\leq 1/4\wedge \a/2$ we
have
\[ \begin{array}{l} \ds{|u^\e(t)-u^\e(s)|_H}\\
\vs \ds{\leq \le|e^{s A}(e^{(t-s)A}x-x)\r|_H+\le|\int_0^t
e^{(t-r)A}F(u^\e(r),v^\e(r))\,dr-\int_0^s e^{(s-r)A}
F(u^\e(r),v^\e(r))\,dr\r|_H}\\
\vs \ds{ \leq
c_{T,\theta}\,(t-s)^{\theta}\,|x|_{2\theta,2}+c_{T,\theta}\,(t-s)^\theta\|F(u^\e,v^\e)\|_{L^2(0,T;H)}.}
\end{array}
\] This implies that for any $\theta\leq 1/4\wedge \a/2$
\begin{equation}
\label{holder} [u^\e]_{C^\theta([0,T];H)}=\sup_{s\neq
t}\,\frac{|u^\e(t)-u^\e(s)|}{|t-s|^{\theta}}\leq
c_T\,\le(|x|_{2\theta,2}+\|F(u^\e,v^\e)\|_{L^2(0,T;H)}\r).\end{equation}
Now, according to estimates \eqref{stima11} and \eqref{stima12},
we have
\begin{equation}
\label{wit28} \E\,\|F(u^\e,v^\e)\|_{L^2(0,T;H)}\leq
c\le(1+\E\,|u^\e|_{L^2([0,T];H)}+\E\,|v^\e|_{L^2([0,T];H)}\r)\leq
c_T\le(1+|x|_H+|y|_H\r),
\end{equation}
and hence
\begin{equation}
\label{wit27} \sup_{\e>0}\,\E\,|u^\e|_{C^{\theta}([0,T];H)}\leq
c_T\le(1+|x|_{2\theta,2}+|y|_H\r).
\end{equation}

Next,  if $\theta<1/2\wedge \a$, thanks to \eqref{wit21} for any $t
\in\,[0,T]$ we have
\[\begin{array}{l}
\ds{|u^\e(t)|_{\theta,2}\leq c_T\,|x|_{\theta,2}+c_T\int_0^t
(t-s)^{-\frac \theta 2}\,|F(u^\e(s),v^\e(s))|_H\,ds}\\
\vs \ds{\leq
c_T\,|x|_{\theta,2}+c_{T,\theta}\,\|F(u^\e,v^\e)\|_{L^2(0,T;H)}.}
\end{array}\] Then, by using again \eqref{wit28}, we have
\begin{equation}
\label{wit29} \sup_{\e>0}\,\E\sup_{t
\in\,[0,T]}|u^\e(t)|_{\theta,2}\leq c_{T,\theta}\le(1+|x|_{\theta,2}+|y|_H\r).
\end{equation}

Combining together \eqref{wit27} and \eqref{wit29}, we conclude
that for any $\eta>0$ there exists $R(\eta)>0$ such that
\[\mathbb{P}\le(u^\e \in\,\mathcal{K}_{R(\eta)}\r)\geq 1-\eta,\ \ \ \ \
\e>0,\] where, by the Ascoli-Arzel\`a theorem, $\mathcal{K}_{R(\eta)}$ is
the compact subset of $C((0,T];H)\cap L^\infty(0,T;H)$ defined by
\[\mathcal{K}_{R(\eta)}:=\le\{\,u \,:\,\sup_{t
\in\,[0,T]}|u(t)|_{\theta,2}\leq R(\eta),\ |u|_{C^{\theta}([0,T];H)}\leq
R(\eta)\,\r\},\] for some $\theta<1/4\wedge \a/2$. This implies the tightness of
the family $\{\mathcal{L}(u^\e)\}_{\e>0}$ in $C((0,T];H)\cap L^\infty(0,T;H)$.

\end{proof}

We conclude the present section by proving that if $x$ and $y$ are taken 
in $W^{\a,2}(0,L)$, for some $\a>0$, then $u^\e(t) \in\,D(A)$, for
$t>0$. Moreover,  we provide an estimate for the momentum of the norm of
$A u^\e(t)$, which is uniform with respect to $\e \in\,(0,1]$.

\begin{Lemma}
\label{lemma4.4}
 Assume that $x, y \in\,W^{\a,2}(0,L)$, for some 
$\a \in\,(0,2]$. Then, under Hypotheses \ref{H1}, \ref{H2} and \ref{H3},
we have that $u^\e(t) \in\,D(A)$, $\mathbb{P}$-{\em a.s.}, for any
$t>0$ and $\e>0$. Moreover,  for any $T>0$ and $\e \in\,(0,1]$ it holds
\begin{equation}
\label{cosimo} \E\,|A u^\e(t)|_H\leq c\,t^{\frac \a
2-1}\,|x|_{\a,2}+c_T\,(1+\e^{-\frac {\a\vee (1-\gamma)}
2})\,(1+|x|_{\a,2}+|y|_{\a,2}),\ \ \ \ t \in\,(0,T],
\end{equation}
where $\gamma$ is the constant introduced in \eqref{h11}.
\end{Lemma}

\begin{proof}
We decompose $u^\e(t)$ as
\[\begin{array}{l}
\ds{u^\e(t)=u^\e_1(t)+u^\e_2(t):=\le[e^{t A} x+\int_0^t e^{(t-s)A}\,F(u^\e(t),v^{\e}(t))\,ds\r]}\\
\vs \ds{+\int_0^t
e^{(t-s)A}\,\le[F(u^\e(s),v^{\e}(s))-F(u^\e(t),v^{\e}(t))\r]\,ds.}
\end{array}\]
According to \eqref{wit21} we have
\[\begin{array}{l}
\ds{|A u^\e_1(t)|_H\leq |A e^{t A} x|_H +|(e^{t
A}-I)F(u^\e(t),v^{\e}(t))|_H}\\
\vs \ds{\leq c_T\,t^{\frac \a
2-1}\,|x|_{\a,2}+c_T\le(1+|u^\e(t)|_H+|v^\e(t)|_H\r),}
\end{array}\]
so that, thanks to \eqref{stima11} and \eqref{stima12}
\begin{equation}
\label{u1} \E\,|A u^\e_1(t)|_H\leq c_T\,t^{\frac \a
2-1}\,|x|_{\a,2}+c_T\le(1+|x|_H+|y|_H\r),\ \ \ \ t \in\,(0,T].
\end{equation}
Concerning $u^\e_2(t)$, we have
\[\begin{array}{l}
\ds{|A u^\e_2(t)|_H\leq c_T\int_0^t
(t-s)^{-1}\le(\,|u^\e(t)-u^\e(s)|_H+|v^\e(t)-v^\e(s)|_H\r)\,ds}\\
\vs \ds{\leq c_T\int_0^t (t-s)^{\frac \a
2-1}\,ds\,[u^\e]_{C^{\frac \a
2}(0,T;H)}+c\int_0^t(t-s)^{-1}|v^\e(t)-v^\e(s)|_H\,ds,}
\end{array}\]
and then, due to \eqref{wit27}, by taking
expectation we have
\begin{equation}
 \label{u2} \E\,|A
u^\e_2(t)|_H\leq
c_T\,(1+|x|_{\a,2}+|y|_H)+c\int_0^t(t-s)^{-1}\E\,|v^\e(t)-v^\e(s)|_H\,ds.
\end{equation}
This means that, in order to conclude the proof, we have to
estimate $\E\,|v^\e(t)-v^\e(s)|_H$, for any $0\leq s<t\leq T$.

It holds
\[\begin{array}{l}
\ds{v^\e(t)-v^\e(s)=\le[e^{t\frac B\e} y-e^{s\frac B \e} y\r]+\frac 1\e
\int_s^t e^{(t-\si)\frac{B}\e}G(u^\e(\si),v^{\e}(\si))\,d\si}\\
\vs \ds{+\frac 1\e \int_0^s
\le[e^{(t-\si)\frac B\e}-e^{(s-\si)\frac B \e}\r]G(u^\e(\si),v^{\e}(\si))\,d\si+\le[w^{\e,B}(t)-w^{\e,B}(s)\r]=:
\sum_{k=1}^4 I^\e_k(t,s).}
\end{array}\]
Proceeding as in the proof of Theorem \ref{tightness}, we have
\begin{equation}
\label{i1} |I^\e_1(t,s)|_H\leq c\,\e^{-\frac \a 2}\,(t-s)^{\frac
\a 2}\,|y|_{\a,2}.
\end{equation}
Concerning $I^\e_2(t,s)$, we have
\[|I^\e_2(t,s)|_H\leq
\frac c\e\int_s^te^{-\la\frac
{(t-\si)}\e}\le(1+|u^\e(\si)|_H+|v^\e(\si)|_H\r)\,d\si,\] and
then, with a change of variables, according to \eqref{stima11} and
\eqref{stima12} we get
\begin{equation}
\label{i2} \E\,|I^\e_2(t,s)|_H\leq c_T\int_0^{\frac{t-s}\e}
e^{-\la \si}\,d\si (1+|x|_H+|y|_H)\leq c_T \e^{-\frac \a
2}\,(t-s)^{\frac \a 2} (1+|x|_H+|y|_H).
\end{equation}
By proceeding with analogous arguments we prove that
\begin{equation}
\label{i3} \E\,|I^\e_3(t,s)|_H\leq c_T\, \e^{-\frac \a
2}\,(t-s)^{\frac \a 2} (1+|x|_H+|y|_H).
\end{equation}
Therefore, it remains to estimate $\E\,|I^\e_4(t,s)|_H$. By
straightforward computations, we have
\[\begin{array}{l}
\ds{\E\,|I^\e_4(t,s)|_H^2=\E\,|w^{\e,B}(t)-w^{\e,B}(s)|_H^2}\\
\vs \ds{=\frac 1\e \int_s^t\|e^{(t-\si)B/\e}\|_2^2\,d\si+\frac 1\e
\int_0^s\|e^{(t-\si)B/\e}-e^{(s-\si)B/\e}\|_2^2\,d\si=:J^\e_1+J^\e_2.}
\end{array}\]
According to \eqref{wit9},  with a change of
variables we have
\[J^\e_1\leq c\int_0^{\frac{t-s}\e} e^{-2\la \si}(\si\wedge
1)^{-\gamma}\,d\si\leq c\, \e^{-(1-\gamma)} (t-s)^{1-\gamma}.\] 
Concerning $J^\e_2$, due to \eqref{wit21}
and \eqref{analytic} for any $\eta \in\,[0,1/2]$ and $s,t>0$
\[\|(e^{t B}-I)e^{sB}\|_{\mathcal{L}(H)}\leq c_\eta\,(t\wedge
1)^{\eta}(s\wedge 1)^{-\eta}.\]
Hence, thanks to \eqref{wit9}, if $0<\eta<1-\gamma$, by
proceeding with the same change of variables
\[\begin{array}{l}
\ds{J^\e_2=\frac 1\e \int_0^s
\le\|[e^{(t-s)B/\e}-I]e^{(s-\si)B/2\e}e^{(s-\si)B/2\e}\r\|_2^2\,d\si}\\
\vs \ds{\leq \frac c\e \le(\frac {t-s}\e\wedge1\r)^{\eta} 
\int_0^s\le(\frac {s-\si}{2\e}\wedge1\r)^{-(\eta+\gamma)}
e^{-\la \frac {s-\sigma}{2\e}}\,d\si\leq c\, \e^{-\eta}
(t-s)^{\eta},}
\end{array}
\]
so that
\begin{equation}
\label{i4} \E\,|I^\e_4(t,s)|_H\leq
\le(\E\,|I^\e_4(t,s)|_H^2\r)^{\frac 12}\leq c\, \e^{-(1-\gamma)}
\le[(t-s)^{1-\gamma}+(t-s)^\eta\r].
\end{equation}
Collecting together \eqref{i1}, \eqref{i2}, \eqref{i3} and
\eqref{i4}, we obtain
\[\E\,|v^\e(t)-v^\e(s)|_H\leq c_T\,\e^{-\frac {\a\vee (1-\gamma)} 2}
\le(1+|x|_H+|y|_{\a,2}\r)\,\le[(t-s)^{\frac \a 2}+(t-s)^{\eta}+(t-s)^{1-\gamma}\r], \]
so that, from \eqref{u2},
\[\begin{array}{l}
\ds{\E\,|A u^\e_2(t)|_H
\leq c_T\,(1+|x|_{\a,2}+|y|_{\a,2})(1+\e^{-\frac {\a\vee (1-\gamma)} 2}).}
\end{array}\]
Together with \eqref{u1}, this yields \eqref{cosimo}.

\end{proof}

\section{The averaging result}
\label{sec5}

Our aim here is proving the main result of the present  paper.
Namely, we are going to prove that for any fixed $T>0$ the
sequence $\{u^\e\}_{\e>0}\subset C((0,T];H)\cap L^\infty(0,T;H)$ converges in
probability to the solution $\bar{u}$ of the averaged equation
\begin{equation}
\label{averaged} du(t)=A u(t)+\bar{F}(u(t)),\ \ \ \ u(0)=x.
\end{equation}

The non-linear coefficient $\bar{F}$ in the equation above is
obtained by averaging the reaction coefficient $F$
appearing in the slow motion equation, with respect to the unique
invariant measure $\mu^x$ of the fast motion equation \eqref{av5},  with frozen slow component $x$. More precisely,
\begin{equation}
\label{wit31} \bar{F}(x):=\int_H F(x,y)\,\mu^x(dy),\ \ \ \ \ x
\in\,H.
\end{equation}
Notice that, as the mapping $y \in\,H\mapsto F(x,y) \in\,H$ is
Lipschitz-continuous, due to \eqref{lip} the integral above is
well defined. Moreover, as $\mu^x$ is ergodic, for any $h \in\,H$
we have \begin{equation} \label{mixing}
\le<\bar{F}(x),h\r>_H=\lim_{t\to \infty}\frac
1t\int_0^t\le<F(x,v^{x,y}(s)),h\r>_H\,ds,\ \ \ \
\mathbb{P}-\text{a.s.} 
\end{equation} 
This implies that  $\bar{F}$
is Lipschitz-continuous. Actually, as $F:H\times H\to H$ is
Lipschitz-continuous (with Lipschitz-constant $L_f$) and
$v^{x,y}(t)$ is differentiable with respect to $x \in\,H$, with
its derivative fulfilling \eqref{finis}, for any $x_1,x_2 \in\,H$
and $t>0$ we have
\[\begin{array}{l}
\ds{\frac 1t\le|\int_0^t
\le<F(x_1,v^{x_1,y}(s))-F(x_2,v^{x_2,y}(s)),h\r>_H\,ds\r|}\\
\vs \ds{\leq \frac{L_f}t\int_0^t
(|x_1-x_2|_H+|v^{x_1,y}(s)-v^{x_2,y}(s)|_H)\,ds|h|_H}\\
\vs \ds{\leq L_f\,|h|_H\,(|x_1-x_2|_H+\sup_{\substack{x,y
 \in\,H\\t\geq 0}}|D_x v^{x,y}(t)|_{\mathcal{L}(H)}|x_1-x_2|_H)\leq c\,(L_f+1)\,|h|_H\,|x_1-x_2|_H.}
\end{array}\]
Therefore, as \eqref{mixing} holds,  we can conclude that
$\bar{F}$ is Lipschitz-continuous, with
\begin{equation}
\label{lipfbar} [\bar{F}]_{\text{\tiny Lip}}\leq
c\,(L_f+1).\end{equation} In particular, we have the following
existence and uniqueness result for the averaged equation.

\begin{Proposition}
\label{existence}
 Under Hypotheses \ref{H1}, \ref{H2} and
\ref{H3}, equation \eqref{averaged} admits a unique mild solution
$\bar{u} \in\,C((0,T];H)\cap L^\infty(0,T;H)$, for any $T>0$ and $p\geq 1$
and  for any initial datum $x \in\,H$.
\end{Proposition}

As far as the differentiability of $\bar{F}$ is concerned, we have
the following result.

\begin{Lemma}
\label{5.1}
 For any $h \in\,L^\infty(0,L)$, the mapping
$\le<\bar{F}(\cdot),h\r>_H:H\to \reals$ is Fr\'echet
differentiable and for any $k \in\,H$
\[\begin{array}{l}
\ds{\le<D\bar{F}(x)k,h\r>_H=\int_H\le<D_x
F(x,y)k,h\r>_H\,\mu^x(dy)+2\int_H
\le<U_x(x,y),k\r>_H\le<F(x,y),h\r>_H\,\mu^x(dy)}\\
\vs \ds{-2\int_H\le<U_x(x,y),k\r>_H\,\mu^x(dy)\,
\int_H\le<F(x,y),h\r>_H\,\mu^x(dy),}
\end{array}\]
where $U_x(x,y)$ is the Fr\'echet derivative of the mapping
$U(\cdot,y):H\to \reals$ introduced in Subsection \ref{3.2}, for $y \in\,L^\infty(0,L)$ fixed.

\end{Lemma}

\begin{proof}
It is immediate to check that for any $y \in\,L^\infty(0,L)$ the
mapping
\[x \in\,H\mapsto U(x,y)=\int_0^1\le<G(x,\theta
y),y\r>_H\,d\theta \in\,\mathbb{R},\] is Fr\'echet differentiable
and for any $k \in\,H$
\[\le<U_x(x,y),k\r>_H=\int_0^1\le<G_x(x,\theta
y)k,y\r>_H\,d\theta,\] where $G_x(x,y)$ is the G\^ateaux
derivative of $G(\cdot,y)$ introduced in Section \ref{sec2}.

Then, if we define
\[V(x,y):=\frac 1{Z(x)}\, \exp\,2 U(x,y),\ \ \ \ \ x,y \in\,H,\]
by straightforward computations, for any $y \in\,L^\infty(0,L)$
the mapping $V(\cdot,y):H\to H$ is differentiable and we have
\begin{equation}
\label{mis50} D_xV(x,y)=2\,V(x,y)\le[U_x(x,y)-\int_H U_x(x,z)
\mu^x(dz)\r]=:2\,V(x,y) H(x,y). \end{equation} Notice that,  as we
are assuming $\partial g/\partial \si_1(\xi,\si)$ to be  uniformly
bounded, we have
\[|U_x(x,y)|_H\leq c\,|y|_H,\ \ \ \ x,y \in \,H,\]
so that thanks to \eqref{wit22} $D_xV(x,y)$ is well defined.

Now, according to \eqref{suppo}, the measure $\mu^x$
is supported on $C([0,L])$, so that
\[\le<\bar{F}(x),h\r>_H=\int_{C([0,L])}\le<F(x,y),h\r>_H\,\mu^x(dy).\]
Hence, if we set $\mu:=\mathcal{N}(0,(-B)^{-1}/2)$, by
differentiating under the sign of integral from \eqref{mis50} we
have
\[\begin{array}{l}
\ds{\le<D\bar{F}(x)k,h\r>_H=\le<D\int_{C([0,L])}\le<F(x,y),h\r>_H\,V(x,y)\,\mu(dy),k\r>_H}\\
\vs
\ds{=\int_{C([0,L])}\le<D_xF(x,y)k,h\r>_H\mu^x(dy)+2\int_{C([0,L])}\le<F(x,y),h\r>_H\,H(x,y)\,\mu^x(dy)}\\
\vs
\ds{=\int_H\le<D_xF(x,y)k,h\r>_H\mu^x(dy)+2\int_H\le<F(x,y),h\r>_H\,H(x,y)\,\mu^x(dy),}
\end{array}\]
and recalling how $H(x,y)$ is defined, we can conclude the proof
of the lemma.
\end{proof}

\bigskip

Now, as $u^\e$ is a mild solution of the slow motion equation (in fact
it is a classical solution, as $u^\e(t) \in\,D(A)$ for any $t>0$,
and estimate \eqref{cosimo} holds),  for any $h \in\,D(A^\star)$
\[\le<u^\e(t),h\r>_H=\le<x,h\r>_H+\int_0^t \le<u^\e(s),A^\star
h\r>_H\,ds+\int_0^t\le<F(u^\e(s),v^\e(s)),h\r>_H\,ds,\ \ \ \ t\geq
0.\] Hence, we have
\begin{equation}
\label{byparts}
 \le<u^\e(t),h\r>_H=\le<x,h\r>_H+\int_0^t
\le<u^\e(s),A^\star
h\r>_H\,ds+\int_0^t\le<\bar{F}(u^\e(s)),h\r>_H\,ds+R_h^\e(t),\ \ \
\ t\geq 0
\end{equation}
where the remainder $R_h^\e(t)$ is clearly given by
\begin{equation}
\label{remainder}
R_h^\e(t):=\int_0^t\le<F(u^\e(s),v^\e(s))-\bar{F}(u^\e(s)),h\r>_H\,ds,\
\ \ \ \ t\geq 0.
\end{equation}

Our purpose is proving that the remainder $R_h^\e(t)$ converges to
zero, as $\e$ goes to zero. We will see that, thanks to Theorem
\ref{tightness}, this will imply the averaging result.

\begin{Lemma}
\label{resto} Assume Hypotheses \ref{H1}, \ref{H2} and \ref{H3}
and fix any $\a>0$. Then, for any $T>0$, $x, y \in\,W^{\a,2}(0,L)$
and any $h \in\,H$
\begin{equation} \label{restobis}
\lim_{\e\to 0}\,\sup_{t \in\,[0,T]}\,\E\,|R_h^\e(t)|=0.
\end{equation}
\end{Lemma}

\begin{proof}
Fix $h \in\,L^\infty(0,L)$. For any $x,y \in\,H$ and  $\e>0$ we
define
\begin{equation}
\label{phieps} \Phi_h^{\e}(x,y):=\int_0^\infty e^{-c(\e)\,
t}\,P^{x}_t\le[\le<F(x,\cdot),h\r>_H-\le<\bar{F}(x),h\r>_H\r](y)\,dt,
\end{equation}
where  $c(\e)$ is some positive constant, depending on $\e>0$, to
be chosen later on. As for any $y,z \in\,H$
\[|\le<F(x,y),h\r>_H-\le<F(x,z),h\r>_H|\leq c\,|y-z|_H\,|h|_H,\]
for some constant $c$ independent of $x \in\,H$, we have that the
mapping
\[y \in\,H\to \le<F(x,y),h\r>_H-\le<\bar{F}(x),h\r>_H \in\,\reals\]
is Lipschitz-continuous and
\begin{equation}
\label{gore}
[\le<F(x,\cdot),h\r>_H-\le<\bar{F}(x),h\r>_H]_{\text{{\tiny
Lip}}}\leq c\,|h|_H. \end{equation}
 According to Theorem
\ref{av14}, this means that the function $\Phi^{\e}_h(x,\cdot)$ is
a strict solution of the problem
\begin{equation}
\label{kolmo}
c(\e)\,\Phi^{\e}_h(x,y)-\mathcal{L}^{x}\Phi^{\e}_h(x,y)=\le<F(x,y),h\r>_H-\le<\bar{F}(x),h\r>_H,\
\ \ \ \ y \in\,D(B). \end{equation}

\medskip

Now, we prove uniform bounds in $\e>0$ for $\Phi^{\e}_h(x,y)$, for
its first derivatives with respect to $y$ and $x$ and for
$\text{Tr}\,[D^2_y\Phi^{\e}_h(x,y)]$. Due to \eqref{mis5bis} and \eqref{gore}, we
have
\[\le|P^{x}_t\le<F(x,\cdot),h\r>_H(y)-\le<\bar{F}(x),h\r>_H\r|\leq
c\,\le(1+|x|_H+|y|_H\r) e^{-\d t}\,|h|_H,\] and then
\begin{equation}
\label{stimauni} |\Phi_h^{\e}(x,y)|\leq c\int_0^\infty
e^{-c(\e)\,t} e^{-\d t}\,dt \le(1+|x|_H+|y|_H\r)\,|h|_H\leq
\frac{c}{\d}\,\le(1+|x|_H+|y|_H\r)\,|h|_H,
\end{equation}
for some constant $c$ independent of $\e>0$.

For the first derivative with respect to $y$,  from
\eqref{deriprima} and \eqref{gore} we get
\[\le[P^{x}_t\le<F(x,\cdot),h\r>_H-\le<\bar{F}(x),h\r>_H\r]_1\leq
c\,e^{-\d t}\,|h|_H,\] and then
\begin{equation}
\label{der1}
\begin{array}{l} \ds{|D_y
\Phi^{\e}_h(x,y)|_H=\le|\int_0^\infty e^{-c(\e)  t}\,D_y
\le[P^{x}_t\le<F(x,\cdot),h\r>_H(y)-\le<\bar{F}(x),h\r>_H\r]\,dt\r|_H}\\
\vs \ds{\leq c\int_0^\infty e^{-c(\e) t}\,e^{-\d t}\,dt\,|h|_H\leq
\frac c\d \,|h|_H,}
\end{array}
\end{equation}
for a constant $c$ independent of $\e>0$.

For the trace of $D^2_y \Phi_h^\e(x,y)$, according to \eqref{eq14}
we have
\[\le|\text{Tr}\,\le[D^2_y
P^x_t\,\le<F(x,\cdot),h\r>_H(y)\r]\r| \leq c\,(t\wedge
1)^{-\rho}\le(1+|x|_H+|y|_H\r)\,|h|_H,\] for some $\rho
<1$, and hence if $c(\e)\leq 1$
\begin{equation}
\label{traccia}
\begin{array}{l} \ds{\le|\text{Tr}\,\le[D^2_y
\Phi^\e(x,y)\r]\r|\leq \int_0^\infty e^{-c(\e)  t}
\le|\text{Tr}\,\le[D^2_y\le(
P^x_t\,\le<F(x,\cdot),h\r>_H(y)-\le<\bar{F}(x),h\r>_H\r)\r]\r|\,dt
}\\
\vs \ds{\leq \int_0^\infty e^{-c(\e)  t}\,(t\wedge
1)^{-\rho}\,dt\,\le(1+|x|_H+|y|_H\r)\,|h|_H\leq \frac
c{c(\e)}\,\le(1+|x|_H+|y|_H\r)\,|h|_H.}
\end{array}\end{equation}

 Next, concerning  the regularity of $\Phi^{\e}_h$ with
respect to $x \in\,H$, we first compute the derivative of the
mapping
\[x \in\,H\mapsto
P^{x}_t\le<F(x,\cdot),h\r>_H(y)=\E\,\le<F(x,v^{x,y}(t)),h\r>_H
\in\,\reals.\] As we are assuming that $h \in\,L^\infty$, we have
that the mappings $\le<F(x,\cdot),h\r>_H$ and
$\le<F(\cdot,y),h\r>_H$ are both Fr\'echet differentiable (see
Section \ref{sec2}). Beside, as shown at the end of  subsection
\ref{3.0}, the process $v^{x,y}$ is differentiable with respect to
$x$. Then, by differentiating above under the sign of integral,
for any $k \in\,H$ we obtain
\[\begin{array}{l}
\ds{\le<D_x\le[ P^{x}_t\le<F(x,\cdot),h\r>_H(y)\r],k\r>_H}\\
\vs \ds{ =\E\,\le<D_x F(x,v^{x,y}(t))k,h\r>_H
+\E\,\le<D_y F(x,v^{x,y}(t))D_x v^{x,y}(t)k,h\r>_H}\\
\vs  \ds{=P^{x}_t\le<D_x F(x,\cdot)k,h\r>_H(y)+\E\,\le<D_y
F(x,v^{x,y}(t))D_x v^{x,y}(t)k,h\r>_H,}
\end{array}\]
so that, thanks to \eqref{stimaderf} and \eqref{finis}
\[|D_x\le[ P^{x}_t\le<F(x,\cdot),h\r>_H(y)\r]|_H\leq c L_f\,|h|_H,\
\ \ \ \ t\geq 0.\] Moreover, as shown in  Lemma \ref{5.1}, the
mapping $x\in\,H\mapsto \le<\bar{F}(x),h\r>_H \in\,\reals$ is
Fr\'echet differentiable and, due to estimate \eqref{lipfbar}, we
have
\[[\le<\bar{F}(x),h\r>_H]_1=[\le<\bar{F}(x),h\r>_H]_{\text{\tiny
Lip}}\leq c (L_f+1)\,|h|_H.\] Therefore
\begin{equation}
\label{der1x}
\begin{array}{l} \ds{|D_x
\Phi^{\e}_h(x,y)|_H=\le|\int_0^\infty e^{-c(\e)  t}\,D_x
\le[P^{x}_t\le<F(x,\cdot),h\r>_H(y)-\le<\bar{F}(x),h\r>_H\r]\,dt\r|_H}\\
\vs \ds{\leq c\int_0^\infty e^{-c(\e) t}\,dt\,(L_f+1)\,|h|_H=
\frac {c (L_f+1)}{c(\e)}\,|h|_H.}
\end{array}
\end{equation}

\medskip

Next, for any $n \in\,\nat$, we define $v^\e_n:=P_n v^\e$, where
$P_n$ the projection of $H$ onto $\le<e_1,\ldots,e_n\r>$ and
$\{e_k\}_{k \in\,\nat}$ is the complete orthonormal system,
introduced in Hypothesis \ref{H2}, which diagonalizes $B$. It is
immediate to check that $v_n^\e$ is a strong solution of equation
\[dv_n^\e(t)=\frac 1\e\,\le[B
v_n^\e(t)+P_n\,G(u^\e(t),v^\e(t))\r]\,dt+\frac
1{\sqrt{\e}}\,P_n\,dw(t),\ \ \ \ \ v^\e_n(0)=P_n y.\] Moreover,
according to Lemma \ref{lemma4.4}, $u^\e$ is a strong solution of
the slow motion equation.

Therefore, we can apply
 It\^o's formula to
$\Phi^{\e}_h(u^\e(t),v^\e_n(t))$ and we get
\[\begin{array}{l}
\ds{\Phi^{\e}_h(u^\e(t),v^\e_n(t))=\Phi^{\e}_h(x,P_ny)
+\int_0^t\le<D_x
\Phi^{\e}_h (u^\e(s),v^\e_n(s)),A u^\e(s)+F(u^\e(s),v^\e(s))\r>_H\,ds}\\
\vs \ds{+\frac 1\e \int_0^t \le<D_y\Phi^{\e}_h
(u^\e(s),v^\e_n(s)),B v^\e_n(s)+P_n G(u^\e(s),v^\e(s))\r>_H\,ds}\\
\vs \ds{+\frac 1{2\e}\int_0^t \text{Tr}\,[P_n\,D^2_y\Phi^{\e}_h
(u^\e(s),v^\e_n(s))]\,ds +\frac 1{\sqrt{\e}}\int_0^t\le<D_y
\Phi^{\e}_h (u^\e(s),v^\e_n(s)),P_n dw(s)\r>_H,}
\end{array}\]
and hence
\begin{equation}
\label{striscia}
\begin{array}{l}
\ds{\Phi^{\e}_h(u^\e(t),v^\e_n(t))=\Phi^{\e}_h(x,P_ny)
+\int_0^t\le<D_x \Phi^{\e}_h
(u^\e(s),v^\e_n(s)),A u^\e(s)+F(u^\e(s),v^\e(s))\r>_H\,ds}\\
\vs \ds{+\frac 1\e \int_0^t \mathcal{L}^{u^\e(s)}\Phi^{\e}_h
(u^\e(s),v^\e_n(s))\,ds+\frac 1{\sqrt{\e}}\int_0^t\le<D_y
\Phi^{\e}_h (u^\e(s),v^\e_n(s)),P_n dw(s)\r>_H}\\
\vs \ds{+\frac 1\e \int_0^t \le<D_y\Phi^{\e}_h
(u^\e(s),v^\e_n(s)),[P_n
G(u^\e(s),v^\e(s))-G(u^\e(s),v^\e_n(s))]\r>_H\,ds}\\
\vs \ds{+\frac 1{2\e}\int_0^t
\text{Tr}\,[(P_n-I)\,D^2_y\Phi^{\e}_h (u^\e(s),v^\e_n(s))]\,ds.}
\end{array}
\end{equation}
Recalling that $\Phi^{\e}_h(x,\cdot)$ is a strict solution of the
elliptic equation \eqref{kolmo},  for any $s\geq 0$ we have
\[\begin{array}{l}
\ds{\mathcal{L}^{u^\e(s)}\Phi^{\e}_h(u^\e(s),v^\e_n(s))=c(\e)\,\Phi^{\e}_h(u^\e(s),v^\e_n(s))-
\le(\le<F(u^\e(s),v^\e_n(s)),h\r>_H-\le<\bar{F}(u^\e(s)),h\r>_H\r).}
\end{array}\]
Then, multiplying both sides of 
\eqref{striscia} by $\e$,
\[\begin{array}{l}
\ds{R_\e(t)=\int_0^t
\le[\le<F(u^\e(s),v^\e(s)),h\r>_H-\le<\bar{F}(u^\e(s)),h\r>_H\r]\,ds=c(\e)
\int_0^t \Phi^{\e}_h
(u^\e(s),v^\e_n(s))\,ds}\\
\vs \ds{+\sqrt{\e}\int_0^t\le<D_y
\Phi^{\e}_h (u^\e(s),v^\e_n(s)),P_n\,dw(s)\r>_H-\e\,\le[\Phi^{\e}_h(u^\e(t),v^\e_n(t))-\Phi^{\e}_h(x,y)\r]}\\
\vs \ds{+\e\,\int_0^t\le<D_x \Phi^{\e}_h (u^\e(s),v^\e_n(s)),A
u^\e(s)+F(u^\e(s),v^\e(s))\r>_H\,ds+H^{n,\e}(t),}
\end{array}\]
where
\[\begin{array}{l}
\ds{H^{n,\e}(t):=\int_0^t \le<D_y\Phi^{\e}_h
(u^\e(s),v^\e_n(s)),[P_n
G(u^\e(s),v^\e(s))-G(u^\e(s),v^\e_n(s))]\r>_H\,ds}\\
\vs \ds{+\frac 1{2}\int_0^t\!\!
\text{Tr}\,[(P_n-I)\,D^2_y\Phi^{\e}_h
(u^\e(s),v^\e_n(s))]\,ds+\!\!\int_0^t\!\!
\le<F(u^\e(s),v^\e(s))-F(u^\e(s),v_n^\e(s)),h\r>_H\,ds.}
\end{array}\]
Thanks to \eqref{stimauni}, \eqref{der1}  and \eqref{der1x}, this yields
\[\begin{array}{l}
\ds{|R^{\e}_h(t)|\leq c \le(\frac{\e}{c(\e)}+c(\e)
\r)\,\int_0^t\le(1+|u^\e(s)|_H+|v^\e(s)|_H+|v^\e_n(s)|_H\r)\,ds
|h|_H}\\
\vs \ds{+c\,\frac {\e}{c(\e)}\int_0^t |A
u^\e(s)|_H\,ds\,|h|_H+\sqrt{\e}\le|\int_0^t\le<D_y \Phi^{\e}_h
(u^\e(s),v^\e_n(s)),P_n\,dw(s)\r>_H\r|}\\
\vs
\ds{+\e\,\le(1+|u^\e(t)|_H+|v^\e_n(t)|_H+|x|_H+|y|_H\r)\,|h|_H+|H^{n,\e}|,}
\end{array}
\]
and hence, by taking expectation, due to \eqref{stima11},
\eqref{stima12}, \eqref{cosimo} and \eqref{der1}, for any $n
\in\,\nat$
\[\begin{array}{l}
\ds{\E\,|R^{\e}_h(t)|\leq c_T \le(\frac{\e}{c(\e)}+c(\e)+\e
\r)\,\le(1+|x|_H+|y|_H\r)\,
|h|_H}\\
\vs \ds{+c_T\,\frac {\e}{c(\e)}(1+\e^{-\frac {\a\vee (1-\gamma)}
2})(1+|x|_{\a,2}+|y|_{\a,2})|h|_H+c_T\,\sqrt{\e}\,|h|_H+\E\,|H^{n,\e}|.}
\end{array}
\]

Now, thanks to estimates \eqref{stima11}, \eqref{stima12},
\eqref{der1} and \eqref{traccia}, by using the dominated
convergence theorem, for any $\e>0$ we have
\[\lim_{n\to \infty}\,\E\,|H^{n,\e}|=0.\]
This means that if we take $c(\e)=\e^\d$, with $0<\d<1-[\a\vee (1-\gamma)]/2$, and
$n_\e \in\,\nat$ such that $\E\,|H^{n_\e,\e}|\leq \e$, it follows
\[\begin{array}{l}
\ds{\sup_{t \in\,[0,T]}\,\E\,|R^{\e}_h(t)|\leq c_T
\le(\frac{\e}{c(\e)}+c(\e)+\e
\r)\,\le(1+|x|_H+|y|_H\r)\, |h|_H}\\
\vs \ds{+c_T\,\frac
{\e}{c(\e)}(1+\e^{-\frac {\a\vee (1-\gamma)} 2})(1+|x|_{\a,2}+|y|_{\a,2})|h|_H+c_T\,\sqrt{\e}\,|h|_H+\E\,|H^{n_\e,\e}|}\\
\vs \ds{ \leq c_T\,e^{\rho}\,(1+|x|_{\a,2}+|y|_{\a,2})|h|_H,}
\end{array}\] for some $\rho>0$. This immediately yields
\eqref{restobis} for $h \in\,L^\infty(0,L)$.

Now, if $h \in\,H$ we fix a sequence $\{h_n\}_{n \in\,\nat}\subset
L^\infty(0,L)$ converging to $h$ in $H$ and such that $|h_h|_H\leq
|h|_H$. As
\[\sup_{t \in\,[0,T]}\,\E\,|R_{h_n}^\e(t)|\leq
c_T\,\e^{\rho}\,(1+|x|_{\a,2}+|y|_{\a,2})|h_n|_H,\] we obtain
\eqref{restobis} also in the general case.

\end{proof}

Once we have proved the key Lemma \ref{resto}, we can prove the
main result of the  paper, the convergence of the solution of the
slow motion equation to the solution of the {\em averaged}
equation.

\begin{Theorem}
\label{averaging} Assume that $x, y \in\,W^{\a,2}(0,L)$, for some
$\a>0$. Then, under Hypotheses \ref{H1}, \ref{H2} and \ref{H3},
for any $T>0$ and $\eta>0$ we have
\begin{equation}
\label{wit30} \lim_{\e\to
0}\mathbb{P}\,\le(\sup_{t \in\,[0,T]}|u^\e(t)-\bar{u}(t)|_{H}>\eta\,\r)=0,
\end{equation}
where $\bar{u}$ is the solution of the averaged equation
\eqref{averaged}.
\end{Theorem}

\begin{proof}
Due to  Theorem \ref{tightness},  the sequence
$\{\mathcal{L}(u^\e)\}_{\e>0}$ is tight in $C((0,T];H)\cap L^\infty(0,T;H)$, and then
as a consequence of the Skorokhod theorem,  for any two sequences
$\{\e_n\}_{n \in\,\nat}$ and $\{\e_m\}_{m \in\,\nat}$ converging
to zero, there exist subsequences $\{\e_{n(k)}\}_{k \in\,\nat}$
and $\{\e_{m(k)}\}_{k \in\,\nat}$ and a sequence of random
elements
\[\{\rho_k\}_{k \in\,\nat}:=\le\{(u_1^k,u_2^k)\r\}_{k \in\,\nat},\]
in $\mathcal{C}:=\le[C((0,T];H)\cap L^\infty(0,T;H)\r]^2$, defined on some probability
space $(\hat{\Omega},\hat{\F},\hat{\Pro})$, such that the law of
$\rho_k$ coincides with the law of
$(u^{\e_{n(k)}},u^{\e_{m(k)}})$, for each $k \in\,\nat$, and
$\rho_k$ converges $\hat{\Pro}$-a.s. to some random element
$\rho:=(u_1,u_2) \in\,\mathcal{C}$. By a well known
argument due to Gy\"ongy and Krylov (see \cite{gk}), if we show
that $u_1=u_2$, then we can conclude that there exists some $u
\in\,C((0,T];H)\cap L^\infty(0,T;H)$ such that the whole sequence $\{u^\e\}_{\e>0}$
converges to $u$ in probability.

For any $k \in\,\nat$ and $i=1,2$ we define
\begin{equation}
\label{fine?} R_i^k(t):=\langle u^k_i(t),h\rangle_H-\langle
x,h\rangle_H-\int_0^t\langle u^k_1(s),A^\star h\rangle_H\,ds
-\int_0^t\langle\bar{F}(u^k_i(s)),h\rangle_H\,ds.\end{equation}
 As
$\mathcal{L}(u^k_1)=\mathcal{L}(u^{\e_{n(k)}})$ and
$\mathcal{L}(u^k_1)=\mathcal{L}(u^{\e_{m(k)}})$, according to
\eqref{restobis} we have
\[\lim_{k\to \infty }\,\sup_{t \in\,[0,T]}\hat{\E}\,|R^k_i(t)|=0,\]
so that, as the sequences $\{u^k_1\}_{k \in\,\nat}$ and
$\{u^k_2\}_{k \in\,\nat}$ converge $\hat{\mathbb{P}}$-a.s. in
$C((0,T];H)\cap L^\infty(0,T;H)$ respectively to $u_1$ and $u_2$, by taking the limit
for some  $\{k_i(n)\}\subseteq \{k\}$ going to infinity in
\eqref{fine?}, we have that both $u_1$ and $u_2$ fulfill the
equation
\[
\langle u(t),h\rangle_H=\langle x,h\rangle_H+\int_0^t \langle
u(s),A^\star h\rangle_H\,ds+\int_0^t\langle
\bar{F}(u(s)),h\rangle_H\,ds,\] for any $h \in\,D(A^\star)$, 
and then they coincide with the unique solution of
the {\em averaged} equation \eqref{averaged}.

As we have recalled before, this implies that the sequence
$\{u^\e\}_{\e>0}$ converges in probability to some $u
\in\,C([0,L];H)$, and, by using again a uniqueness argument, such $u$ has to coincide with the solution $\bar{u}$ of equation \eqref{averaged}.

\end{proof}

\section{Some remarks on the case of space dimension $d>1$}
\label{sec6}

In the case of space dimension $d=1$, the fast equation \eqref{av5} with frozen slow component $x \in\,H$ is a gradient system and hence its unique invariant measure $\mu^x$ admits a density $V(x,y)$ with respect to the Gaussian measure $\cal{N}(0,(-B)^{-1}/2)$. This allows to prove in Lemma \ref{5.1} that for any $h \in\,L^\infty(0,L)$ the mapping 
\begin{equation}
\label{fbar}
x \in\,H\mapsto \le<\bar{F}(x),h\r>_H \in\,\reals
\end{equation}
 is Fr\'echet differentiable and also allows to compute its derivative.

In space dimension $d>1$, in order to have function-valued solutions to system \eqref{sistema} we have to take a noise colored in space, and hence the fast equation is no more a gradient system.  For this reason we cannot say anything about the differentiability of mapping \eqref{fbar} and hence we cannot say anything about the differentiability with respect to $x \in\,H$ of the mapping $\Phi^\e_h(x,y)$ introduced in \eqref{phieps}.
Nevertheless, under suitable assumptions on the noise in the fast equation, it is possible to prove a result analogous to that proved in Lemma \ref{resto} and hence to get averaging.

Instead of working in the interval $(0,L)$, now we work in a bounded open set $D\subset \reals^d$, with $d>1$, having a regular boundary. In the fast motion equation we take a noise of the following form
\[w^Q(t,\xi)=\sum_{k=1}^\infty Q e_k(\xi)\beta_k(t),\ \ \ \ t\geq 0,\ \ \ \ \xi \in\,D,\]
and we assume that the operators $B$ and $Q$ satisfy the following conditions.

\begin{Hypothesis}
\label{H1bis}

\begin{enumerate}
\item There exists a complete orthonormal system $\{e_k\}_{k
\in\,\nat}$ in $H$ and two positive sequences $\{\a_k\}_{k
\in\,\nat}$ and $\{\la_k\}_{k \in\,\nat}$ such that
$B e_k=-\a_k e_k$ and $Q e_k=\la_k e_k$
and,  for some
$\gamma<1$,
\[
 \sum_{k=1}^\infty \frac
{\la_k^2}{\a_k^{1-\gamma}}<\infty.
\]
\item There exists $\la>0$ such that
$\a_k\geq \la$, for any $k \in\,\nat$.
\item There exists $\eta<1/2$ such that
\[\inf_{k \in\,\nat}\la_k \a_k^\eta>0.\]
\end{enumerate}

\end{Hypothesis}
Notice that, as $\a_k\sim k^{2/d}$, the conditions above imply that we have to work with $d\leq 3$.
Under Hypothesis \ref{H1bis} and  Hypotheses \ref{H2} and \ref{H3} (with obvious changes due to the passage from $d=1$ to $d\geq 1$) system \eqref{sistema} admits a unique mild solution $(u^\e,v^\e) \in\,\cal{C}_{T,p}\times \cal{C}_{T,p}$, for any $\e>0$, $p\geq 1$ and $T>0,$ and for any fixed slow component $x \in\,H$ the fast equation \eqref{av5} admits a unique mild solution $v^{x,y} \in\,\cal{C}_{T,p}$, fulfilling \eqref{suppo} and \eqref{wit1bis}.  As in the one dimensional case, the process $v^{x,y}$ is three times differentiable with respect to $y \in\,H$ and once with respect to $x \in\,H$ and estimates analogous to 
\eqref{av7}, \eqref{av8} and \eqref{finis} hold (for all details see \cite{cerrai}).

The fast transition semigroup $P^x_t$ maps $C_b(H)$ into itself and $\text{Lip}(H)$ into itself and \eqref{lipptbis} holds. Moreover, it has a smoothing effect and maps $B_b(H)$ into $C_b^3(H)$ and estimates \eqref{eq13}, \eqref{deriprima} and \eqref{derisucc} are still true, with the singularity $(t\wedge 1)^{(j-i)/2}$ replaced by $(t\wedge 1)^{(j-i)(\eta+1/2)}$.

As far as the asymptotic behavior of the fast semigroup is concerned, it admits a unique invariant measure $\mu^x$ which is strongly mixing and fulfills \eqref{wit22}, \eqref{mis5} (with the singularity $(t\wedge 1)^{1/2}$ replaced by the singularity  $(t\wedge 1)^{-(\eta+1/2)}$) and \eqref{mis5bis}. But, as we have said before, as equation \eqref{av5} is not of gradient type,  we do not have any explicit expression for  the measure $\mu^x$.

All uniform bounds for $u^\e$ and $v^\e$ proved in Section \ref{sec4} are still valid, so that the family of probability measures $\{\cal{L}(u^\e)\}_{\e \in\,(0,1)}$ is tight in $C((0,T];H)\cap L^\infty(0,T;H)$. This means that in order to have averaging in this multidimensional case it suffices to prove Lemma \ref{resto}. The proof in this case  follows the same lines as in the one dimensional case, but it requires some extra approximation arguments. 
 Actually, one has to introduce 
the approximating problems
\begin{equation}
\label{appne} dv^\e_n(t)=\frac 1\e \le[B_n
v^\e_n+G_n(u^\e(t),v^\e_n(t))\r]\,dt+\frac 1{\sqrt{\e}}\,Q_n
\,dw(t),\ \ \ \ v^\e_n(0)=P_n y,
\end{equation}
and
\begin{equation}
\label{appn} dv_n^{x,y}(t)=\le[B_n
v_n^{x,y}+G_n(x,v^{x,y}_n(t))\r]\,dt+Q_n \,dw(t),\ \ \ \
v^{x,y}_n(0)=P_n y,
\end{equation}
where $B_n x:=BP_n x$, $Q_n x:=Q P_n x$ and  $G_n(x,y):=P_n G(x,P_nx)$, for any $n \in\,\nat$ and $x,y \in\,H$.
As the operators $B_n$ and $Q_n$ fulfill Hypothesis \ref{H1bis} and
$G_n$ has the same regularity properties of $G$, all properties
satisfied by $v^\e$, $v^{x,y}$ and $P^x_t$ are still valid
for $v_n^\e$, $v_n^{x,y}$ and for the transition semigroup
$P^{n,x}_t$ associated with \eqref{appn}.
Moreover, all estimates for $v^{x,y}_n$ and
$P^{n,x}_t$ are uniform with respect to $n \in\,\nat$, and for 
each fixed $\e>0$ and $x,y \in\,H$
\begin{equation}
\label{limye} \lim_{n\to\infty}\E\,\sup_{t\geq
0}|v^\e_n(t)-v^\e(t)|_H^2=0,
\end{equation}
and
\begin{equation}
\label{limyn} \lim_{n\to\infty}\E\,\sup_{t\geq
0}|v^{x,y}_n(t)-v^{x,y}(t)|_H^2=0.
\end{equation}
Clearly, equation \eqref{appn} shows the same long-time behavior
as equation \eqref{av5}. Then for any $n \in\,\nat$ there exists a
unique invariant measure $\mu^{n,x}$ for the semigroup
$P^{n,x}_t$, which fulfills all properties described  for $\mu^x$, with all estimates uniform with respect to
$n \in\,\nat$.

Next, we define
\[\bar{F}_n(x):=\int_HF(x,y)\,\mu^{n,x}(dy),\ \ \ \ \ x \in\,H.\] 
As for $\bar{F}$, we obtain that all $\bar{F}_n:H\to H$ are Lipschitz-continuous and  
\begin{equation}
\label{bis30} \sup_{n \in\,\nat}\,[\bar{F}_n]_{\text{Lip}}\leq
c\,L_f.
\end{equation}
Moreover, for any $x \in\,H$
\begin{equation}
\label{limfbar} \lim_{n\to
\infty}\le|\bar{F}_n(x)-\bar{F}(x)\r|_H=0.
\end{equation}
For any $n \in\,\nat$ we define
\[ H_n(x):=\int_{\mathbb{R}^n}\bar{F}_n(P_n
x-\sum_{k=1}^n \xi_k e_k)\rho_n(\xi)\,d\xi,\ \ \ \ \ \ x
\in\,H,\]
 where 
$\rho_n:\mathbb{R}^n\to\reals$ is a $C^1$ mapping having support
in $B_{\mathbb{R}^n}(0,1/n)$ and having total mass equal $1$. All
mappings $H_n$ are in $C^1(H;H)$ and
\begin{equation}
\label{effen} \lim_{n\to \infty}\,|\bar{F}_n(x)-H_n(x)|_H=0,\ \ \
\ \ x \in\,H.
\end{equation}
Moreover, due to \eqref{bis30}, we have
\begin{equation}
\label{bis20} |\bar{F}_n(x)-H_n(x)|_H\leq c\,\le(1+|x|_H\r),\ \ \
\ \ \sup_{n \in\,\nat}\,[H_n]_{\text{Lip}(H)}<\infty.
\end{equation}

Then, in the proof of Lemma \ref{resto} we introduce the following correction function 
\[\Phi_n^{\e}(x,y):=\int_0^\infty e^{-c(\e)\,
t}\,P^{n,x}_t\le[\le<F(x,\cdot),h\r>_H-\le<H_n(x),h\r>_H\r](y)\,dt,
\]
As in the one dimensional case, we have that the function $\Phi^{\e}_n(x,\cdot)$ is
a strict solution of the problem
\[c(\e)\,\Phi^{\e}_n(x,y)-\mathcal{L}^{n,x}\Phi^{\e}_n(x,y)=\le<F(x,y),h\r>_H-\le<H_n(x),h\r>_H,\
\ \ \ \ y \in\,H, \]
where $\mathcal{L}^{n,x}$ is the Kolmogorov operator associated with the approximating fast equation \eqref{appn}.

Concerning the regularity of $\Phi^{\e}_n$ with respect to $y$, we proceed as in the proof of Lemma \ref{resto} and all estimates are uniform with respect to $n \in\,\nat$. As far as regularity in $x$ is concerned, we also proceed as in the one dimensional case, by noticing that the mapping $x\in\,H\mapsto
\le<H_n(x),h\r>_H \in\,\reals$ is Fr\'echet differentiable and,
due to estimate \eqref{effen}, the $C^1$-norm is uniformly bounded in $n \in\,\nat$, that is
\[\sup_{n \in\,\nat}\,[\le<H_n(\cdot),h\r>_H]_1=\sup_{n \in\,\nat}\,[\le<H_n(\cdot),h\r>_H]_{\text{\tiny
Lip}}=c\,|h|_H<\infty.\]
This implies an estimate for $D_x\Phi_n^\e$, which is  uniform with respect to $n \in\,\nat$.

Next, as in the proof of Lemma \ref{resto} we apply It\^o's formula to
$\Phi^{\e}_n(u^\e(t),v^\e_n(t))$ and, by some estimates not different from those already used, by taking $c(\e)=\e^\delta$, for some $\delta>0$ we obtain 
\[\begin{array}{l}
\ds{\E\,\le|\int_0^t
\le[\le<F(u^\e(s),v_n^\e(s)),h\r>_H-\le<H_n(u^\e(s)),h\r>_H\r]\,ds\r|}\\
\vs \ds{\leq c_T \e^{\d^\prime}
\le(1+|x|_{\a,2}+|y|_{\a,2}\r)\,|h|_H+\int_0^T\E\,\le|\bar{F}_n(u^\e(s))-H_n(u^\e(s))\r|_H\,ds\,|h|_H,}
\end{array}\]
for some $\d^\prime>0$. Due to \eqref{limfbar} and \eqref{effen}, this allow to conclude that \eqref{restobis} holds.

\end{document}